\documentclass[opre,nonblindrev]{informs3}

\DoubleSpacedXI 

\usepackage{endnotes}
\let\footnote=\endnote
\let\enotesize=\normalsize
\def\notesname{Endnotes}%
\def\makeenmark{\hbox to1.275em{\theenmark.\enskip\hss}}
\def\enoteformat{\rightskip0pt\leftskip0pt\parindent=1.275em
  \leavevmode\llap{\makeenmark}}

\usepackage{booktabs}       
\usepackage{amsmath}
\usepackage{amssymb}
\usepackage{xfrac}
\usepackage{algorithm}
\usepackage{acronym}
\usepackage{algorithmic}

\usepackage{graphicx}
\usepackage{caption}
\usepackage{subcaption}
\usepackage{hyperref}
\hypersetup{
    colorlinks=false,
    pdfborder={0 0 0},
}
\usepackage{tikz}
\usetikzlibrary{calc,patterns,arrows,shapes.arrows,intersections}

\usepackage{natbib}
 \bibpunct[, ]{(}{)}{,}{a}{}{,}%
 \def\bibfont{\small}%
\TheoremsNumberedThrough     
\ECRepeatTheorems

\EquationsNumberedThrough    

\newcommand{\norm}[1]{\left\|#1\right\|}

\usepackage[space]{grffile} 
\usepackage{etoolbox} 
\makeatletter 
\patchcmd\Gread@eps{\@inputcheck#1 }{\@inputcheck"#1"\relax}{}{}
\makeatother
\usepackage{color}

\usepackage{comment}

\begin{document}


\RUNAUTHOR{Bertsimas, Pauphilet and Van Parys}

\RUNTITLE{Sparse Classification: A Scalable Discrete Optimization Perspective}

\ARTICLEAUTHORS{%

\AUTHOR{Dimitris Bertsimas}
\AFF{Sloan School of Management and Operations Research Center, MIT, Cambridge, MA, USA, \EMAIL{dbertsim@mit.edu}}
\AUTHOR{Jean Pauphilet}
\AFF{Operations Research Center, MIT, Cambridge, MA, USA, \EMAIL{jpauph@mit.edu}}
\AUTHOR{Bart Van Parys}
\AFF{Operations Research Center, Sloan School of Management, MIT, Cambridge, MA, USA, \EMAIL{vanparys@mit.edu}}

} 

\TITLE{Sparse Classification: A Scalable Discrete Optimization Perspective}


\ABSTRACT{
We formulate the sparse classification problem of $n$ samples with  $p$ features as a binary convex optimization problem and propose a cutting-plane algorithm to solve it exactly. For sparse logistic regression and sparse SVM, our algorithm finds optimal solutions for $n$ and $p$ in the $10,000$s within minutes. On synthetic data our algorithm achieves perfect support recovery in the large sample regime. Namely, there exists a $n_0$ such that the algorithm takes a long time to find the optimal solution and does not recover the correct support for $n<n_0$, while for $n\geqslant n_0$, the algorithm quickly detects all the true features, and does not return any false features. In contrast, while Lasso accurately detects all the true features, it persistently returns incorrect features, even as the number of observations increases. Consequently, on numerous real-world experiments, our outer-approximation algorithms returns sparser classifiers  
while achieving similar predictive accuracy as Lasso. 
To support our observations, we analyze conditions on the sample size needed to ensure full support recovery in classification. Under some assumptions on the data generating process, we prove that information-theoretic limitations impose $n_0 <  C \left(2 + \sigma^2\right) k \log(p-k)$, for some constant $C>0$.
}


\SUBJECTCLASS{Programming: integer, nonlinear; Statistics: data analysis}

\maketitle
\section{Introduction}
\label{sec:intro}
Sparse classification is a central problem in machine learning as it leads to more interpretable models. Given data $\{(x_i, y_i)\}_{i=1,\dots,n}$ with $y_i \in \{-1,1\}$ and $x_i \in \mathbb{R}^p$, we aim at computing the estimator $w$ which minimizes an empirical loss $\ell$ subject to the constraint that its number of nonzero entries does not exceed $k$: 
\begin{equation}
\label{eqn:sparse_reg}
\min_{w \in \mathbb{R}^p,b \in \mathbb{R}} \sum_{i=1}^n \ell(y_i, w^\top x_i +b) \mbox{  s.t.  } \Vert w \Vert_0 \leqslant k.
\end{equation}
Despite its conceptual appeal, Problem \eqref{eqn:sparse_reg} is recognized as an NP-hard optimization problem \citep{natarajan1995sparse}. Thus, much of the literature has focused on heuristic proxies and replaced the $\ell_0$ pseudo-norm with so-called sparsity-inducing convex norms \citep{bach2012optimization}. Even though regularization enforces robustness more than sparsity \citep{bertsimas2009equivalence}, the $\ell_1$-penalty formulation
\begin{equation}
\label{eqn:lasso}
\min_{w \in \mathbb{R}^p,b \in \mathbb{R}} \sum_{i=1}^n \ell(y_i, w^\top x_i +b) + \lambda \Vert w \Vert_1,
\end{equation}
known as Lasso \citep{tibshirani} is abundantly used in practice. Efficient numerical algorithms exist \citep{friedman2010regularization}, off-the-shelf implementations are publicly available \citep{friedman2013glmnet} and recovery of the true sparsity is theoretically guaranteed under some assumptions on the data. In regression problems with i.i.d. Gaussian measurements for instance, \citet{wainwright2009sharp} proved that Lasso recovers the $k$ correct features with high probability (w.h.p.\ in short) for $n>(2k + \sigma^2) \log p$ where  $\sigma^2$ is the variance of the noise, a phenomenon they refer to as phase transition in accuracy. On the other hand, recent works \citep{fan2010sure,tibshirani2011regression,su2015false} highlighted the difficulty for $\ell_1$-regularized estimators to select correct features without making false discoveries, considering Lasso as a good feature screening but a poor feature selection procedure. 

Besides algorithm-specific performance, any support recovery algorithm faces information-theoretic limitations as well \citep{wainwright2009information,wang2010information}. In regression, recent work \citep{gamarnik17} indeed proved the existence of a sharp information-theoretic threshold $n^\star$: If $n<n^\star$, exact support recovery by any algorithm is impossible, while it is theoretically achievable for $n>n^\star$. Such results call for further research in learning algorithms in the regime $n^\star < n < (2k + \sigma^2) \log p$ where Lasso fails but  full recovery is achievable  in principle. 

New research in numerical algorithms for solving the exact sparse formulation \eqref{eqn:sparse_reg} has flourished and demonstrated significant improvement on existing heuristics. \citet{bertsimas2014statistics,bertsimas2015logistic} made use of recent advances in mixed-integer optimization to solve sparse linear and logistic regression problems. \citet{pilanci2015sparse} applied a Lagrangian relaxation and random rounding procedure for linear regression and provide sufficient conditions for support recovery with their method. {  \citet{hazimeh2018fast} developed a cyclic coordinate descent strategy combined with local search to efficiently find local optima of the $\ell_0$-penalized ordinary least square regression problem, later extended to other loss functions \citep{dedieu2020learning}.} Recently, sparse linear regression for $n$ and $p$ in $100,000$s was exactly solved for the first time, using a cutting-plane algorithm \citep{bertsimas2016cio}. Their method demonstrates a clear phase transition in accuracy as the sample size $n$ increases and requires less data than Lasso to achieve full recovery. Simultaneously, they observed a phase transition in false discovery, that is, the number of incorrect features selected, and in computational time, which is unique to their method: they exhibited a threshold $n_0$ for which for $n<n_0$ their algorithm takes a long time to find the optimal solution and it is does not recover the  correct support, while for $n\geqslant n_0$, the algorithm is very fast and accurately detects all the true features, but does not return any false features.

Regarding accuracy, such phase transition phenomena are actually not specific to regression problems but are observed in many data analysis and signal processing contexts \citep{donoho2006breakdown,donoho2009observed}. Surprisingly, little if no work focused on classification problems specifically. Guarantees in terms of $\ell_2$ error have been obtained in the so-called $1$-bit compressed sensing setting \citep{boufounos20081, gupta2010sample, plan2013one, jacques2013robust} but they do not precisely address the question of support recovery. {  In recent work, \citet{scarlett2017limits} offer a comprehensive treatment of information theoretic limitations in support recovery for both linear and $1$-bit compressed sensing, in some regimes of noise and sparsity. } Sparse classification in itself has mainly been regarded as a feature selection problem \citep{dash1997feature} and greedy procedures such as Recursive Feature Elimination \citep{guyon2002gene} have shown the most successful.

In this paper, we formulate exact sparse classification as a binary convex optimization problem, propose a tractable cutting-plane algorithm {  and a stochastic variant} to solve it in high dimensions. We also provide an information-theoretic sufficient condition for support recovery in classification that complements existing results in the literature for 1-bit compressed sensing.

\paragraph{Contributions} The contributions of the present paper can be summarized as follows:
\begin{enumerate}
\item Based on duality results for regularized classification, we formulate the exact sparse classification problem as a binary convex optimization problem and propose a tractable outer-approximation algorithm to solve it. Our approach generalizes the one in \citet{bertsimas2016cio}, who address linear regression for which a closed-form solution exists and make extensive use of this closed-form solution. Our framework, however, extends to cases where a closed-form solution is not available and includes, in addition to linear regression, logistic regression and SVM with one or two norms. {  We also propose a stochastic cut generating process to improve scalability of the outer-approximation algorithm with respect to the number of samples $n$.}
\item We demonstrate the tractability and relevance of our algorithm in solving large binary classification problems with logistic and Hinge loss. Among others, we solve a real-world gene-selection problem with $n=1,145$, $p=14,858$ and select four to ten times fewer genes than the $\ell_1$ heuristic with little compromise on the predictive power.
{  On synthetic data, our algorithm can scale to data sets with up to $p=50,000$ features and prove optimality in less than a minute for low sparsity regimes ($k=5$). Our present algorithm and the concurrent paper of \citet{dedieu2020learning}
are, to the best of our knowledge, the only methods currently available that solve sparse classification problems to provable optimality in such high dimensions within minutes. We believe, however, that the tools developed in the present paper may be more versatile and broadly applicable than the tailored ``integrality generation'' technique of \citet{dedieu2020learning}.
Finally, we observe that our stochastic cut generating process reduces computational time by a factor $2$-$10$ compared with the standard outer-approximation procedure for the hardest instances, namely instances where the number of samples $n$ is not large enough for Lasso to achieve perfect recovery.}
\item We demonstrate empirically that cardinality constrained estimators asymptotically achieve perfect support recovery on synthetic data: As $n$ increases, the method accurately detects all the true features, just like Lasso, but does not return any false features, whereas Lasso does. In addition, we observe that computational time of our algorithm decreases as more data is available. We exhibit a smooth transition towards perfect support recovery in classification settings, whereas it is empirically observed \citep{donoho2006breakdown,bertsimas2016cio} and theoretically defined \citep{wainwright2009information,gamarnik17} as a sharp phase transition in the case of linear regression. 
\item Intrigued by {  this smooth} empirical behavior, we show that there exists a threshold $n_0$ on the number of samples such that if $n > n_0$, the underlying truth $w^\star$ minimizes the empirical error with high probability. Assuming $p \geqslant 2k$ and data is generated by  $y_i = sign\left(x_i^\top w^{\star} + \varepsilon_i\right), $ with  $x_i$ i.i.d. Gaussian, $w^\star \in \{0,1\}^p$, supp$(w^{\star} )=k$ and $\varepsilon_i\sim \mathcal{N} (0, \sigma^2)$, we prove that $n_0 <  C \left(2 + \sigma^2\right) k \log(p-k)$, for some constant $C>0$ (Theorem \ref{sufficient}). 
{  Our information-theoretic sufficient condition on support recovery parallels the one obtained by \citet{wainwright2009information} in regression settings, although discreteness of the outputs substantially modify the analysis and the scaling of the bound. }
In recent work, \citet{scarlett2017limits} performed a similar analysis of support recovery in 1-bit compressed sensing. In low signal-to-noise settings, they exhibit matching necessary and sufficient conditions, which suggests the existence of a phase transition. In high signal-to-noise regimes, however, they proved necessary conditions only. Our result is novel, in that it is not only valid for specific sparsity and noise regimes, but for all values of $k$ and $\sigma$ (see Table \ref{tab:comparison.scarlett} for a summary of their results and ours). In particular, our sufficient condition holds when the sparsity $k$ scales linearly in the dimension $p$ and the signal-to-noise ratio is high. In this regime, there is a $\sqrt{\log p}$ factor between our bound and the necessary condition from \citet{scarlett2017limits}. As represented in Figure \ref{fig:theory.summary}, there is an intermediate sample size regime where support recovery is neither provably impossible nor provably achievable, and could explained the smooth phase transition observed. Such an observation is made possible by the combination of both our results. Of course, this observation only suggests that a lack of phase transition is plausible. So is the existence of one and future work might improve on these bounds and close this gap.
\end{enumerate}

 \begin{table}
 \TABLE
{Summary of the necessary and sufficient conditions provided in \citet{scarlett2017limits}, as compared to the sufficient condition we provide in Theorem \ref{sufficient}. We refer to \citep{scarlett2017limits} for the definition of $f_3(\ell)$. \label{tab:comparison.scarlett}}
{
 \begin{tabular}{lc|c}
 \toprule
Parameters  &$k = \Theta(1)$, Low SNR &$k = \Theta(p)$, High SNR  \\
 \midrule
&\multicolumn{2}{c}{Necessary condition for  $\mathbb{P}(\text{error})\not\rightarrow 1$} \\
\citep{scarlett2017limits} & $\pi \sigma^2 \log p$ & $\Omega(p \sqrt{\log p})$\\
 \midrule
&\multicolumn{2}{c}{Sufficient condition for  $\mathbb{P}(\text{error})\rightarrow 0$} \\
\citep{scarlett2017limits} & $\pi \sigma^2 \log p$ & $-$ \\
  Theorem \ref{sufficient} & \multicolumn{2}{c}{  $C (2+\sigma^2) k \log(p-k)$} \\
 \bottomrule
 \end{tabular}
 }
 {}
 \end{table}

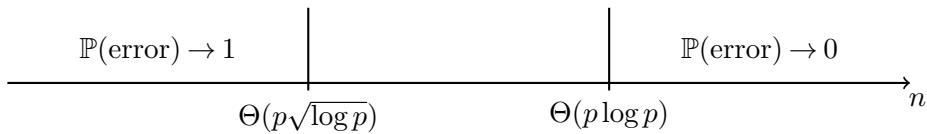
\begin{figure}
\caption{Summary of known necessary \citep[][Cor. 4]{scarlett2017limits} and sufficient (see Theorem \ref{sufficient}) conditions on the sample size $n$ to achieve perfect support recovery in classification, when the sparsity $k$ scales linearly in the dimension $p$ and the signal-to-noise ratio is high. Thresholds are given up to a multiplicative constant.} \label{fig:theory.summary}
\centering
{ 
\begin{tikzpicture}
   \draw[thick, ->] (0, 0) -- (12, 0);
    \node[below] at (12.1, 0) {$n$};

	\draw[thick, -] (4, -.1) -- (4, 1);
    \node[below] at (4, -.1) {$\Theta(p\sqrt{\log p})$};
    \node[above] at (2, .1) {$\mathbb{P}(\text{error})\rightarrow 1$};

    \draw[thick, -] (8, -.1) -- (8, 1);
    \node[below] at (8, -.1) {$\Theta(p \log p)$};
    \node[above] at (10, .1) {$\mathbb{P}(\text{error})\rightarrow 0$};
 \end{tikzpicture}
 }
\end{figure}

{  \paragraph{Structure} We derive the binary convex optimization formulation for exact sparse classification, present an outer-approximation algorithm to solve it, and propose a stochastic cut generating process in Section \ref{sec:dual_framework}. We evaluate the performance of our method, both in terms of tractability and support recovery ability, in Section \ref{sec:log_reg}. Finally, we prove information-theoretic sufficient condition for support recovery in Section \ref{sec:theory}.}

\paragraph{Notation} We denote by $\textbf{e}$ the vector whose components are equal to one. If not specified, its dimension should be inferred from the context. The set $S_k^p$ denotes the set
$$S_k^p := \left\lbrace s\in \lbrace 0, 1 \rbrace^p \: : \: \textbf{e}^\top s \leqslant k \right\rbrace ,$$ which contains all binary vectors $s$ selecting $k$ components from $p$ possibilities. Assume $(y_1,\ldots,y_p)$ is a collection of elements and $s \in S_k^p$, then $y_s$ denotes the sub-collection of $y_j$ where $s_j=1$. We use $\| x \|_0$ to denote the number of elements of a vector which are nonzero. 

\section{Dual Framework}
\label{sec:dual_framework}
In this section, we use duality results to formulate a regularized version of the sparse classification problem \eqref{eqn:sparse_reg} as a binary convex optimization problem and propose a cutting-plane approach to solve it efficiently.
\subsection{Regularized Classification}
We first introduce the basic notation and recall some well-known properties for the case of non-sparse classification. Two very popular classification methods in machine learning are logistic regression and Support Vector Machine (SVM). Despite different motivations and underlying intuition, both methods lead to a similar formulation which can be addressed under the unifying lens of regularized classification:   
\begin{equation}
\label{eqn:reg_class}
\min_{w \in \mathbb{R}^p,b \in \mathbb{R}} \sum_{i=1}^n \ell(y_i, w^\top x_i +b) + \dfrac{1}{2 \gamma} \Vert w \Vert_2^2,
\end{equation}
where $\ell$ is an appropriate loss function and $\gamma$ a regularization coefficient.

In the logistic regression framework, the loss function is the logistic loss $$\ell(y, u) = \log\left( 1+e^{-yu} \right),$$ and the objective function can be interpreted as the negative log-likelihood of the data plus a regularization term, which ensures strict convexity of the objective and existence of an optimal solution. 

In the SVM framework, the loss function $\ell$ is the hinge loss: $$\ell(y, u) = \max(0, 1-yu).$$ Under proper normalization assumptions, the square norm $\| w \|_2^2$ relates to the notion of margin, which characterizes the robustness of the separating hyperplane $\lbrace x : w^\top x + b = 0 \rbrace$, while the loss part penalizes the data points which do no satisfy $y_i(w^\top x_i + b) \geqslant 1$, that is points which are misclassified or lie within the margin \citep{vapnik1998support}.

In addition, this general formulation \eqref{eqn:reg_class} can be extended to a broad family of other loss functions used in classification (e.g. 2-norm SVM) or even in regression problems. Throughout the paper we make the following assumption:
\begin{assumption}
\label{conv_loss}
The loss function $\ell(y,\cdot)$ is convex for $y \in \{-1,1\}$.
\end{assumption}
In classification, deeper results and insights can typically be obtained by adopting a dual perspective. Denoting $X = (x_i^\top)_{i=1,...,n} \in \mathbb{R}^{n \times p}$ the design matrix, we have:
\begin{theorem}
\label{dual_nonsparse}
Under Assumption \ref{conv_loss}, strong duality holds for problem \eqref{eqn:reg_class} and its dual is
\begin{equation}
\label{eqn:dual_reg_class}
\max_{\alpha \in \mathbb{R}^n : \textbf{e}^\top \alpha = 0} - \sum_{i=1}^n  \hat \ell(y_i,\alpha_i) - \dfrac{\gamma}{2} \alpha^\top X X^\top \alpha \; ,
\end{equation} 
where $ \hat \ell(y, \alpha) := \max_{u \in \mathbb{R}} u \alpha - \ell(y,u) $ is the \emph{Fenchel conjugate} of the loss function $\ell$ \citep[see][chap.~3.3]{boyd2004convex}.
\end{theorem}
Table \ref{tab:loss_functions} summarizes some popular loss functions in classification and their corresponding Fenchel conjugates. 
\begin{table}[b]
\TABLE
{Examples of loss functions and their corresponding Fenchel conjugates, as defined in Theorem \ref{dual_nonsparse} \citep{bach2009high}.\label{tab:loss_functions}}
{\begin{tabular}{lll}
\toprule
Method &Loss $\ell(y, u)$ & Fenchel conjugate $\hat \ell(y,\alpha)$ \\
\midrule
Logistic loss & $\log \left( 1 + e^{-y u}\right)$ & $\left\{  \begin{array}{ll} (1 + y \alpha) \log ( 1 + y \alpha) - y \alpha \log ( - y \alpha), &\mbox{ if } y \alpha \in [-1,0],  \\  \vspace{3pt}  +\infty, & \mbox{ otherwise.} \end{array}  \right. $\\ \midrule
1-norm SVM & $\max(0, 1-y u)$ &   $\left\{ \begin{array}{ll} y \alpha, &\mbox{ if } y \alpha \in [-1,0],  \\   \vspace{3pt}  +\infty, & \mbox{ otherwise.} \end{array}\right.$  \\ \midrule
2-norm SVM & $\tfrac{1}{2} \max(0, 1-y u)^2$ &$\left\{ \begin{array}{ll}\tfrac{1}{2} \alpha^2 + y \alpha, &\mbox{ if } y \alpha \leqslant 0,  \\  \vspace{3pt}  +\infty, &\mbox{ otherwise.} \end{array}\right. $ \\
\bottomrule
\end{tabular}}
{}
\end{table}
\proof{Proof of Theorem \ref{dual_nonsparse}}
For regularized classification, we have that
\begin{align*}
\min_{w,b} \sum_{i=1}^n \ell(y_i, w^\top x_i +b) + \dfrac{1}{2 \gamma} \Vert w \Vert_2^2 =\min_{w,b, z} \sum_{i=1}^n \ell(y_i, z_i) + \dfrac{1}{2 \gamma} \Vert w \Vert_2^2 \mbox{ s.t. } z_i = w^\top x_i + b. \end{align*}
By Assumption \ref{conv_loss}, the objective is convex, the optimization set is convex and Slater's conditions hold \citep{boyd2004convex}. Hence, strong duality must hold and the primal is equivalent to the dual problem. To derive the dual formulation, we introduce Lagrange multipliers $\alpha_i$ associated with the equality constraints: 
\begin{flalign*}
&\min_{w,b, z} \sum_{i=1}^n \ell(y_i, z_i) + \dfrac{1}{2 \gamma} \Vert w \Vert_2^2 \mbox{ s.t. } z_i = w^\top x_i + b \\
&=\min_{w,b, z} \sum_{i=1}^n \ell(y_i, z_i) + \dfrac{1}{2 \gamma} \Vert w \Vert_2^2 +  \max_{\alpha \in \mathbb{R}^n} \sum_{i=1}^n \alpha_i (w^\top x_i + b - z_i)\\
&=\min_{w,b, z} \max_{\alpha}   \left( \sum_{i=1}^n  \ell(y_i, z_i) - \alpha_i z_i \right) + \left( \dfrac{1}{2 \gamma} \Vert w \Vert^2 + w^\top \left[ \sum_i \alpha_i x_i \right] \right) + b \, \textbf{e}^\top \alpha \\
&=\max_{\alpha} \sum_{i=1}^n \min_{z_i} \left( \ell(y_i, z_i) - \alpha_i z_i \right) + \min_{w} \left( \dfrac{1}{2 \gamma} \Vert w \Vert^2 + w^\top X^\top \alpha \right) + \min_{b} b \, \textbf{e}^\top \alpha.
\end{flalign*}
Let us consider the three inner minimization problems separately. 
\noindent First, 
\begin{align*}
\min_{z_i} \left( \ell(y_i, z_i) - \alpha_i z_i \right) &= -\max_{z_i} \left( \alpha_i z_i - \ell(y_i, z_i) \right) = -\hat \ell(y_i, \alpha_i). 
\end{align*}
Then, $\tfrac{1}{2 \gamma} \Vert w \Vert^2 + w^\top X^\top \alpha $ is minimized at $w^\star$ satisfying: $ \tfrac{1}{\gamma} w^\star + X^\top \alpha = 0$. Hence \begin{align*}
\min_{w} \left( \dfrac{1}{2 \gamma} \Vert w \Vert^2 + w^\top X^\top \alpha \right) &= - \dfrac{1}{2 \gamma} \Vert w^\star \Vert^2 = - \dfrac{\gamma}{2} \alpha^\top XX^\top \alpha. 
\end{align*}
Finally, $\min_{b} b \, \textbf{e}^\top \alpha$ is bounded if and only if $\textbf{e}^\top \alpha = 0$, thus we obtain \eqref{eqn:dual_reg_class}. \hfill\halmos
\endproof

The derivation of the dual \eqref{eqn:dual_reg_class} reveals that the optimal primal variables $w^\star$ can be recovered from the dual variables $\alpha^\star$ via the relationship $w^\star = - \gamma X^\top \alpha^\star$. In other words, $w^\star$ is a linear combination of the data points $X$. Such an observation has historically led to the intuition that $w^\star$ was \emph{supported} by some observed vectors $x_i$ and the name Support Vector Machine was coined \citep{cortes1995support}. {  Conversely, $\alpha^\star$  relates to the primal variables $w^\star$ via the relationship $\alpha^\star_i \in \partial \ell(y_i, x_i^\top w^\star)$, where $\partial \ell(y_i, x_i^\top w^\star)$ denotes the sub-differential of the loss function $\ell(y_i,\cdot)$ evaluated at $x_i^\top w^\star$. If $\ell$ is differentiable, like the logistic loss, this relationship uniquely defines $\alpha^\star_i$. }

Moreover, the dual point of view opens the door to non-linear classification using kernels \citep{scholkopf2001learning}. The positive semi-definite matrix $X X^\top$, often referred to as the kernel or Gram matrix, is central in the dual problem \eqref{eqn:dual_reg_class} and could be replaced by any kernel matrix $K$ whose entries $K_{ij}$ encode some measure of similarity between inputs $x_i$ and $x_j$.


\paragraph{Numerical algorithms} There is a rich literature on numerical algorithms for solving either the primal \eqref{eqn:reg_class} or the dual \eqref{eqn:dual_reg_class} formulation for the regularized classification problem in the case of logistic regression and SVM. Gradient descent or Newton-Raphson methods are well-suited when the loss function is smooth. In addition, in the case where the dual problem is constrained to $\textbf{e}^\top \alpha = 0$, particular step size rules \citep{calamai1987projected,bertsekas1982projected} or trust regions \citep{lin2008trust} can be implemented to cope with such linear constraints. When the loss function is not continuously differentiable, sub-gradient descent as proposed in the Pegasos algorithm \citep{shalev2011pegasos} provides an efficient optimization procedure. Among the machine learning community, coordinate descent methods have also received a lot of attention recently, especially in the context of regularized prediction, because of their ability to compute a whole regularization path at a low computational cost \citep{friedman2010regularization}. For coordinate descent algorithms specific to the regularized classification problem we address in this paper, we refer to \citep{hsieh2008dual,yu2011dual,keerthi2005fast}.

\begin{remark}
Theorem \ref{dual_nonsparse} does not hold  when the loss is not convex, for instance in the case of the empirical misclassification rate 
\begin{equation}
\label{eqn:sparse_empiricalrate}
\min_{w,b} \sum_{i=1}^n \textbf{1}_{y_i( w^\top x_i +b) < 0} + \dfrac{1}{2 \gamma} \Vert w \Vert_2^2.
\end{equation}
Indeed, strong duality does not hold. The objective value is clearly finite and nonnegative but since $\hat \ell(y, \alpha) = +\infty$, the dual problem has cost $-\infty$.
\end{remark}

\subsection{Dual Approach to Sparse Classification}
Sparsity is a highly desirable property for statistical estimators, especially in high-dimensional regimes ($p\gg n$) such as the ones encountered in biological applications, where interpretability is crucial. A natural way to induce sparsity is to add a constraint on the number of nonzero coefficients of $w$ and solve:
\begin{equation}
\label{eqn:sparse_reg_class}
\min_{w \in \mathbb{R}^p,b \in \mathbb{R}} \sum_{i=1}^n \ell(y_i, w^\top x_i +b) + \dfrac{1}{2 \gamma} \Vert w \Vert_2^2 \; \mbox{  s.t. } \; \Vert w \Vert_0 \leqslant k.
\end{equation}
Actually, \eqref{eqn:sparse_reg_class} can be expressed as a convex binary optimization problem as stated in the following theorem:

\begin{theorem}
\label{cio_formulation}
Problem \eqref{eqn:sparse_reg_class} is equivalent to 
\begin{equation}
\label{eqn:cio}
\min_{s \in S_k^p} c(s),
\end{equation}
where for any $s \in \{0,1\}^p$, 
\begin{equation}
\label{eqn:dual_inner_reg}
c(s) := \max_{\alpha \in \mathbb{R}^n: \textbf{e}^\top\alpha = 0} f(\alpha,s), \; \mbox{  with  } f(\alpha,s) :=  -\sum_{i=1}^n \hat \ell(y_i, \alpha_i) - \dfrac{\gamma}{2} \sum_{j=1}^n s_j \alpha^\top X_j X_j^\top \alpha .
\end{equation}
In particular, $c(s)$ is convex over $[0,1]^p$.
\end{theorem}

\proof{Proof of Theorem \ref{cio_formulation}}
Similarly to \citet{bertsimas2016cio}, we introduce an additional binary variable $s \in \{0,1\}^p$ encoding for the support of the sparse classifier $w$. With these notations 
\begin{align*}
w^\top x_i &= \sum_{j=1}^p w_j x_{i,j} = \sum_{j: s_j =1} w_j x_{i,j} =  w_s^\top x_{is}, \\
\Vert w \Vert_2^2 &= \sum_{j=1}^p w_j^2 = \sum_{j: s_j =1} w_j^2 = \Vert w_s \Vert_2^2,
\end{align*}
the cardinality constraint on $w$ yields a linear constraint on $s$
\begin{align*}
s^\top \textbf{e} \leqslant k,
\end{align*}
and \eqref{eqn:sparse_reg_class} can be equivalently written as
\begin{align*}
\min_{s \in S_k^p} \min_{w,b} \sum_{i=1}^n \ell(y_i, w_s^\top x_{is} +b) + \dfrac{1}{2 \gamma} \Vert w_s \Vert_2^2.
\end{align*}
Denoting $c(s)$ the inner minimization problem, we end up solving the pure binary to-be-proved-convex optimization problem
\begin{align*}
\min_{s \in S_k^p} c(s).
\end{align*}
In addition, $c(s) = \min_{w,b} \sum_{i=1}^n \ell(y_i, w_s^\top x_{is} +b) + \dfrac{1}{2 \gamma} \Vert w_s \Vert_2^2 $ is an unconstrained regularized classification problem based on the features selected by $s$ only. Hence, Theorem \ref{dual_nonsparse} applies and 
\begin{align*}
c(s) &= \max_{\alpha \in \mathbb{R}^n} \sum_{i=1}^n -\hat \ell(y_i, \alpha_i) - \dfrac{\gamma}{2} \alpha^\top X_s X_s^\top \alpha \; \mbox{   s.t.  } \; \textbf{e}^\top\alpha = 0.
\end{align*}
Since $\alpha^\top X_s X_s^\top \alpha = \sum_{j:s_j =1} \alpha^\top X_j X_j^\top \alpha = \sum_{j=1}^p s_j \alpha^\top X_j X_j^\top \alpha$, we obtain the desired formulation. 

Finally, let us denote $$f(\alpha, s) :=-\sum_{i=1}^n \hat \ell(y_i, \alpha_i) - \dfrac{\gamma}{2} \sum_{j=1}^n s_j \alpha^\top X_j X_j^\top \alpha.$$ The function $f$ is convex - indeed linear - in $s$ over $[0,1]^p$, so $c$ is convex over $[0,1]^p$. \hfill\halmos
\endproof

In practice, for a given support $s$, we evaluate the function $c(s)$ by solving the maximization problem \eqref{eqn:dual_inner_reg} with any of the numerical procedures presented in the previous section. In what follows, we need to calculate a sub-gradient of the function $c$ as well. Using the dual maximizer $\alpha^\star(s)$ in \eqref{eqn:dual_inner_reg} at a support $s$, we can compute one at no additional computational cost. Indeed, it follows that
\[
	\dfrac{\partial c(s) }{\partial s_j} = - \dfrac{\gamma}{2} \alpha^{\star}(s)^\top X_j X_j^\top \alpha^\star(s).
\]

\subsection{Enhanced Cutting-plane Algorithm}
We solve the convex binary optimizaton problem \eqref{eqn:cio} taking into account that we can readily compute $c(s)$ and $\nabla c (s)$ for any given $s$. None of the commercial solvers available are targeted to solve such CIO problems where there is no closed-form expression for $c(s)$. We propose to adopt an outer approximation approach similar to the one introduced by \citet{duran1986outer} for linear mixed-integer optimization problems.

We first reformulate \eqref{eqn:cio} as a mixed-integer optimization problem in epigraph form
\begin{equation}
\label{eqn:cmio}
\min_{s \in S_k^p, \eta} \eta \mbox{  s.t.  } \eta \geqslant c(s).
\end{equation}
As described in \citet{fletcher1994solving,bonami2008algorithmic}, we find a solution to \eqref{eqn:cmio} by iteratively constructing a piece-wise linear lower approximation of $c$. The solver structure is given in pseudocode in Algorithm \ref{OA}.
\begin{algorithm*}
\caption{Outer-approximation algorithm}
\label{OA}
\fontsize{10}{10}\selectfont
\begin{algorithmic}
\REQUIRE $X \in \mathbb{R}^{n \times p}$, $Y \in \lbrace -1, 1 \rbrace^p$, $k \in \lbrace 1,...,p \rbrace$ , $\gamma$
\STATE $s^{(1)} \leftarrow$ warm-start
\STATE $\eta^{(1)} \leftarrow 0 $
\STATE $t \leftarrow 1 $
\REPEAT
\STATE $s^{(t+1)},\eta^{(t+1)}\leftarrow \text{argmin}_{s,\eta} \left\lbrace \eta \: : s \in S_k^p, \eta \geqslant c(s^{(i)}) + \nabla c(s^{(i)})^\top(s-s^{(i)}),~  i =1,\dots , t \right\rbrace $ 
\STATE $t \leftarrow t+1 $
\UNTIL{$\eta^{(t)} < c(s^{(t)} )$}
\RETURN $s^{(t)} $ 
\end{algorithmic}
\end{algorithm*}

A proof of termination and convergence can be found in \citep{fletcher1994solving}.
\begin{theorem}
\label{termination}
\citep{fletcher1994solving} Under Assumption \ref{conv_loss}, Algorithm \ref{OA} terminates in a finite number of steps and returns an optimal solution of \eqref{eqn:cio}.
\end{theorem}

{  Several enhancements have been proposed to improve the convergence speed of Algorithm \ref{OA}. First, } in its original form, Algorithm \ref{OA} requires solving a mixed-integer linear optimization problem at each iteration, which is computationally inefficient. Modern solvers however, such as Gurobi \citep{gurobi} or IBM CPLEX \citep{cplex2011cplex}, can handle \emph{lazy constraints}, a feature that integrates the cutting-plane procedure within a unique branch-and-bound enumeration tree, shared by all subproblems. We implemted Algorithm \ref{OA} in this fashion. {  In addition, decomposition schemes as Algorithm \ref{OA} benefit from performing a rich root node analysis, as advocated by  \citet{fischetti2017redesigning}. In essence, a ``rich'' root node analysis consists of a good initial feasible solution $s^{(1)}$ (i.e, a good upper bound) and a set of initial constraints of the form $\eta \geqslant c(s^{(i)}) + \nabla c(s^{(i)})^\top(s-s^{(i)})$ to obtain tight lower bound as well. Regarding the warm-start $s^{(1)} $, we recommend using the Lasso estimator provided by the \verb|glmnet| package \citep{friedman2013glmnet} or the solution of the Boolean relaxation of \eqref{eqn:cio}. We refer to \citet{pilanci2015sparse} for a theoretical analysis of the latter, and \citet{bertsimas2019sparse,atamturk2019rank} for efficient numerical algorithms to solve it. Regarding the lower-bound and the initial constraints, \citet{fischetti2017redesigning} suggests using the cuts obtained from solving 
the Boolean relaxation of \eqref{eqn:cio} via a similar outer-approximation procedure - in which case there are no binary variables and the technique is often referred to as Kelley's algorithm \citep{kelley1960cutting}. We refrain from implementing this strategy in our case. Indeed, computing $c(s)$ and $\nabla c(s)$ reduces to solving a binary classification problem over the $k$ features encoded by the support of $s$. As a result, the scalability of our approach largely relies on the fact that cuts are computed for sparse vectors $s$. When solving the Boolean relaxation, however, $s$ can be dense while satisfying the constraint $s^\top \textbf{e} \leqslant k$. Instead, we compute a regularization path for the Lasso estimator up to a sparsity level of $k+1$ using the \verb|glmnet| package, and initialize the outer-approximation with the cuts obtained from these solutions.} 

{  Finally, we propose a stochastic version of the cutting-plane algorithm to improve the scalability with respect to $n$.  At the incumbent solution $s^{(t)}$, we observe that we do not need to solve Problem \eqref{eqn:dual_inner_reg} to optimality to obtain a valid linear lower-approximation of  $c(s)$. Indeed, any $\alpha \in \mathbb{R}^n \: : \: \textbf{e}^\top \alpha = 0$ yields 
\begin{align*}
\label{eqn:dual_inner_reg}
c(s) \geqslant  -\sum_{i=1}^n \hat \ell(y_i, \alpha_i) - \dfrac{\gamma}{2} \sum_{j=1}^n s_j \alpha^\top X_j X_j^\top \alpha.
\end{align*} 
Hence, we propose a strategy  (Algorithm \ref{alg:sto}) to find a good candidate solution $\alpha$ and the corresponding lower approximation $c(s) \geqslant \Tilde{c}(s^{(t)}) + \nabla\Tilde{c}(s^{(t)})^\top(s - s^{(t)})$. Our strategy relies on the fact the primal formulation in $w \in \mathbb{R}^k$ only involves $k$ decision variables, with  $k < n$ or even $k \ll n$ in practice. As a result, one should not need the entire data set to estimate  $w^\star(s^{(t)})$. Instead, we randomly select $bSize$ out of $n$ observations, estimate $w^\star(s^{(t)})$ on this reduced data set, and finally average the result over $B$ subsamples. Typically, we take $bSize = \max( 10\% n, 2k)$ and $B=10$ in our experiments. Then, we estimate $\alpha \in \mathbb{R}^n$ by solving the first-order optimality conditions $\alpha_i \in \partial \ell(y_i, x_i^\top w)$. Since our objective is to generate a cut that tightens the current lower-approximation of  $c(s^{(t)})$, we compare $\Tilde{c}(s^{(t)})$ with our current estimate of $c(s^{(t)})$, $\eta^{(t)}$. If $\eta^{(t)} >  \Tilde{c}(s^{(t)})$, i.e., if the approximate cut does not improve the approximation, we reject it and compute $c(s^{(t)}), \nabla c(s^{(t)}))$ exactly instead. 

\begin{algorithm*}
\caption{Stochastic cut generation procedure}
\label{alg:sto}
\fontsize{10}{10}\selectfont
\begin{algorithmic}
\REQUIRE $s \in \{0,1\}^p$, $X \in \mathbb{R}^{n \times k}$, $Y \in \lbrace -1, 1 \rbrace^p$, $B$, $bSize$
\FOR{$b = 1,\dots, B $}
\STATE{Define $\Tilde{X}^b \in \mathbb{R}^{bSize \times k}$ obtained from $X$ by randomly selecting $bSize$ rows.}
\STATE{Compute $\Tilde{w}^b$, solution of \eqref{eqn:reg_class} with input data $(\Tilde{X}^b, Y)$.}
\ENDFOR
\STATE $w \leftarrow \tfrac{1}{B}\sum_b \Tilde{w}_b$ 
\STATE Find $\alpha_i \in \partial \ell(y_i, x_i^\top w)$ 
\STATE $\Tilde{c} \leftarrow  -\sum_i \hat \ell(y_i, \alpha_i) -\tfrac{\gamma}{2} \sum_j s_j (X_j^\top \alpha)^2 $, $\nabla\Tilde{c}(s)_j \leftarrow  -\tfrac{\gamma}{2} (X_j^\top \alpha)^2$
\RETURN $\Tilde{c}(s),  \nabla\Tilde{c}(s)$
\end{algorithmic}
\end{algorithm*}

}
In the next section, we provide numerical evidence that Algorithm \ref{OA}, both for logistic and hinge loss functions, is a scalable method to compute cardinality constrained classifiers and select more accurately features than $\ell_1$-based heuristics.  

\subsection{Practical implementation considerations} \label{ssec:cv.detail}
Sparse regularized classification \eqref{eqn:sparse_reg_class} involves two hyperparameters - a regularization factor $\gamma$ and the sparsity $k$. For the regularization parameter $\gamma$, we fit its value using cross-validation among values uniformly distributed in the log-space: we start with a low value $\gamma_0$  - typically $\gamma_0$ scaling as  $1 / \max_i \| x_i \|^2$ as suggested in \citep{chu2015warm} - and inflate it iteratively by a factor two. {  We similarly tune $k$ by simple hold-out cross-validation over a range of values. We use out-of-sample Area Under the receiving operator Curve as the validation criterion (to maximize). Although we did not implement it, \cite{kenney2018efficient} present a general blueprint consisting of warm-start strategies and a bisection search procedure that can tangibly accelerate the overall cross-validation loop.}

\section{  Numerical experiments: Scalability and Support Recovery}
\label{sec:log_reg}
{  In this section, we evaluate the numerical performance of our method both in terms of scalability and quality of the features selected.}

The computational tests were performed on a computer with Xeon @2.3GhZ processors, 1 core, 8GB RAM. Algorithms were implemented in Julia {  1.0} \citep{LubinDunningIJOC}, a technical computing language, and Problem \eqref{eqn:cio} was solved with Gurobi {  8.1.0} \citep{gurobi}. We interfaced Julia with Gurobi using the Julia package JuMP {  0.21.1} \citep{dunning2015jump}. {  Since our sparse regularized formulation \eqref{eqn:sparse_reg_class} comprises 2 hyperparameters, $\gamma$ and $k$, that control the degree of regularization and sparsity respectively, we compare our method with the ElasticNet formulation 
\begin{equation}
\label{eqn:enet}
\min_{w \in \mathbb{R}^p,b \in \mathbb{R}} \sum_{i=1}^n \ell(y_i, w^\top x_i +b) + \lambda \left[ (1-\alpha) \Vert w \Vert_1 + \alpha \Vert w \Vert_2^2 \right],
\end{equation}
which similarly contains 2 hyper parameters and can be computed efficiently by the \verb|glmnet| package \citep{friedman2013glmnet}. For a fair comparison, we cross-validate $\lambda$ and $\alpha$ using the same procedure as $\gamma$ and $k$ described in Section \ref{ssec:cv.detail}.}
\subsection{{  Support recovery }on synthetic data}
{  We first consider synthesized data sets to assess the feature selection ability of our method compared it to a state-of-the-art $\ell_1$-based estimators.}
\subsubsection{Methodology} We draw $x_i \sim \mathcal{N}(0_p, \Sigma), i=1,\dots,n$ independent realizations from a $p$-dimensional normal distribution with mean $0_p$ and covariance matrix $\Sigma_{ij} = \rho^{|i-j|}$. Columns of $X$ are then normalized to have zero mean and unit variance. We randomly sample a weight vector $w_{true} \in \lbrace -1, 0, 1 \rbrace$ with exactly $k$ nonzero coefficients. We draw $\varepsilon_i, i=1,\dots,n,$ i.i.d. noise components from a normal distribution scaled according to a chosen signal-to-noise ratio $$\sqrt{SNR} = \| X w_{true} \|_2 / \| \varepsilon \|_2.$$ Finally, we construct $y_i$ 
as
$y_i = \text{sign} \left( w_{true}^\top x_i + \varepsilon_i > 0\right).$
This methodology enables us to produce synthetic data sets where we control the sample size $n$, feature size $p$, sparsity $k$, feature correlation $\rho$ and signal-to-noise ratio $SNR$. 

\subsubsection{Support recovery metrics} Given the true classifier $w_{true}$ of sparsity $k_{true}$, we assess the correctness of a classifier $w$ of sparsity $k$ by its accuracy, i.e., the number of true features it selects
$$A(w) = |\lbrace j : w_j \neq 0, w_{true,j} \neq 0 \rbrace| \in \{0,\ldots,k_{true}\},$$
and the false discovery, i.e., the number of false features it incorporates
$$F(w) = |\lbrace j : w_j \neq 0, w_{true,j} = 0 \rbrace| \in \{ 0,\ldots	,p\}.$$
Obviously,  {  $A(w) + F(w) = |\lbrace j : w_j \neq 0 \rbrace|  = k$}. A classifier $w$ is said to \emph{perfectly recover} the true support if it selects the truth ($A(w)=k_{true}$) and nothing but the truth ($F(w)=0$ or equivalently $k = k_{true}$). 

\subsubsection{Selecting the truth...}
We first compare the performance of our algorithm for sparse regression with a ElasticNet, when both methods are given the true number of features in the support $k_{true}$. As mentioned in the introduction, a key property in this context for any best subset selection method, is that it selects the true support as sample size increases, as represented in Figure \ref{fig:LRkfix1}. From that perspective, both methods demonstrate a similar convergence: As $n$ increases, both classifiers end up selecting the truth, with Algorithm \ref{OA} needing somewhat smaller number of samples than Lasso.

\begin{figure}[h]
\FIGURE
{\includegraphics[width=.7\linewidth]{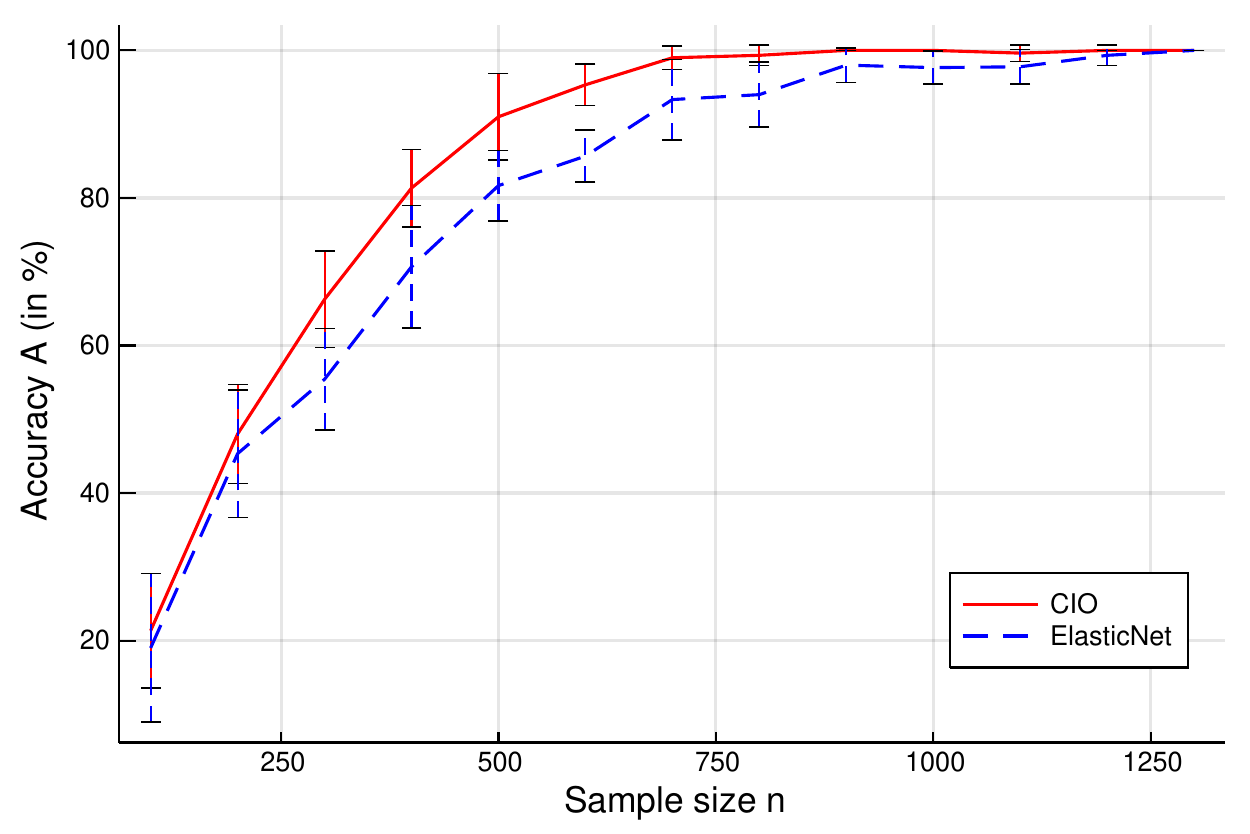}}
{Evolution of the accuracy (number of true features selected) as sample size $n$ increases, for ElasticNet with the logistic loss (dashed blue) and sparse SVM (solid red). \label{fig:LRkfix1}} 
{Results correspond to average values obtained over $10$ data sets with $p=1,000$, $k_{true}=30$, $\rho=0.3$, $SNR \rightarrow \infty$ and increasing $n$ from $100$ to $1,300$.}
\end{figure}

Apart from the absolute number of true/false features selected, one might wonder whether the features selected are actually good features in terms of predictive power. In this metric, sparse regression significantly outperforms the $\ell_1$-based classifier, both in terms of Area Under the Curve (AUC) and misclassification rate, as shown on Figure \ref{fig:LRkfix2}, demonstrating a clear predictive edge of exact sparse formulation. 
\begin{figure*}[h]
\FIGURE
{
\includegraphics[width=.45\linewidth]{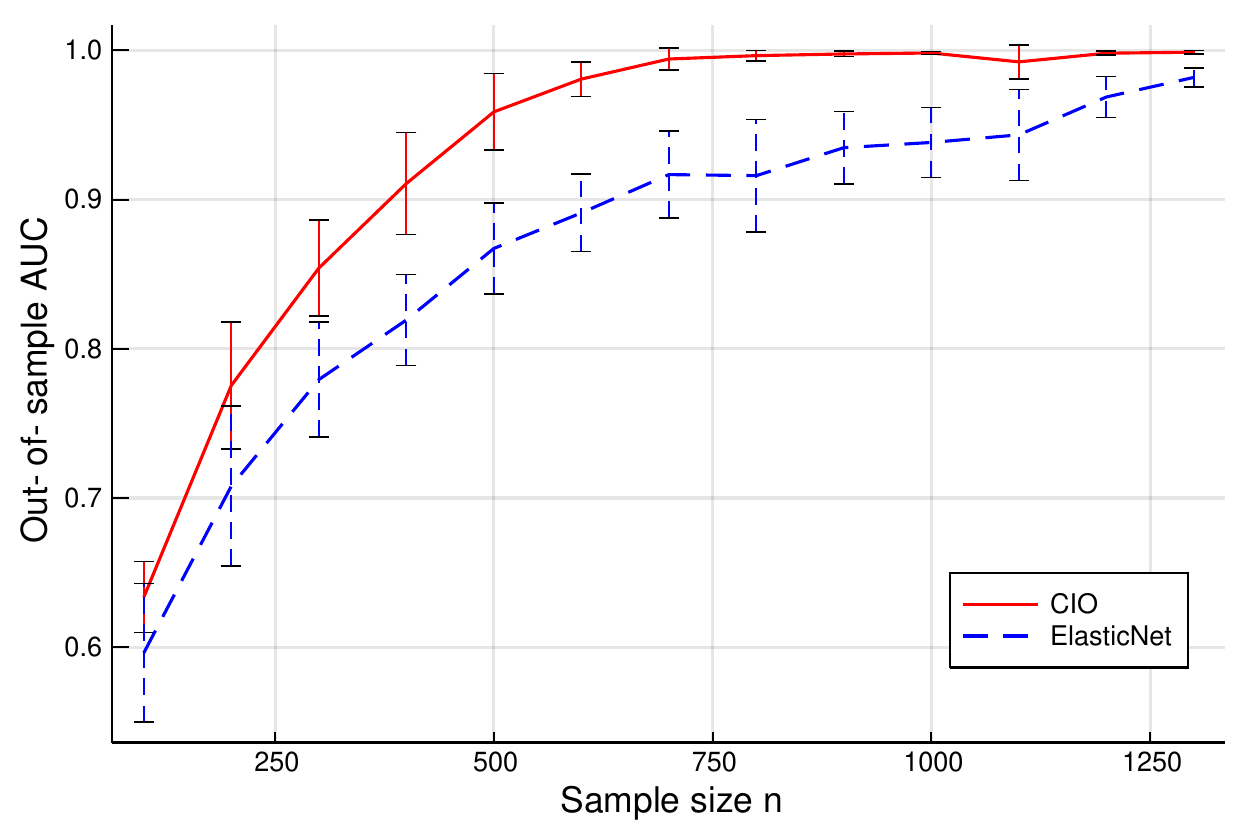}
\includegraphics[width=.45\linewidth]{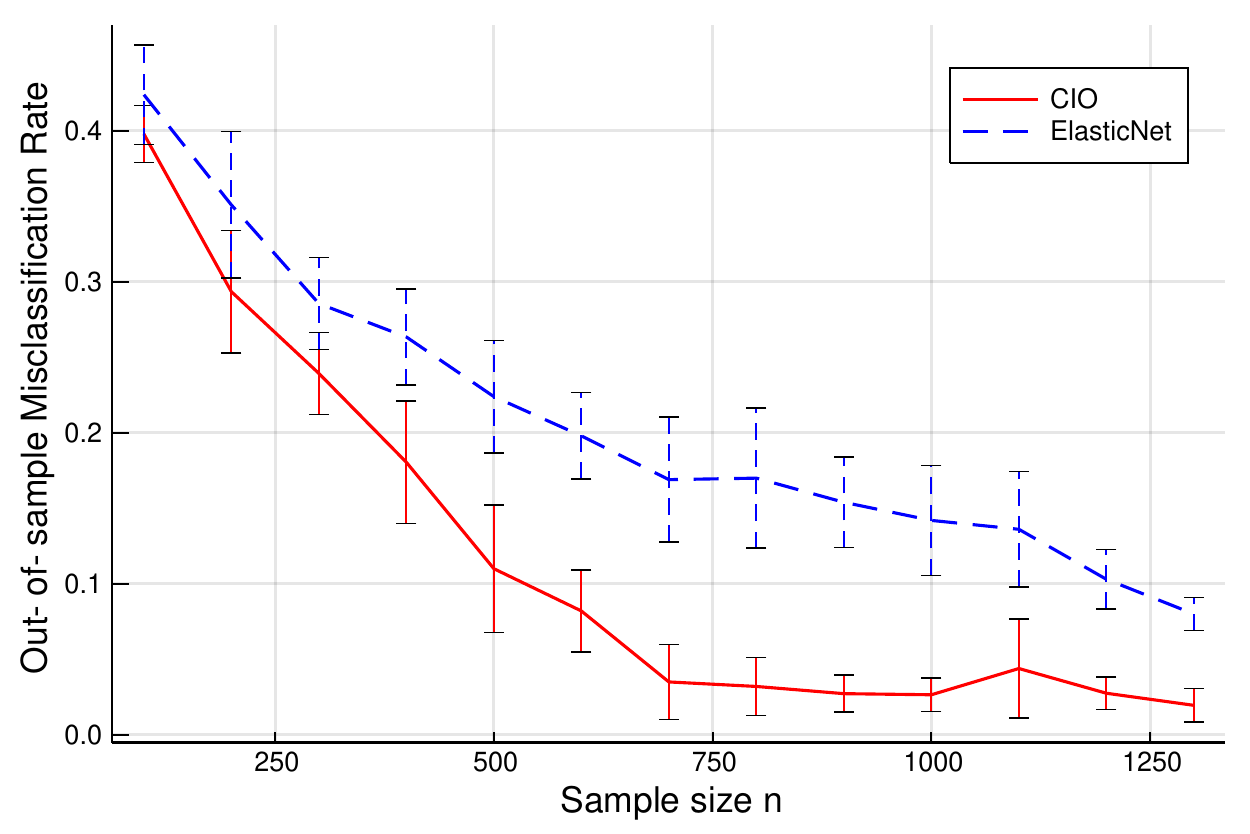}
}
{Evolution of  the AUC (left) and misclassification rate (right)  on a validation set as sample size $n$ increases, for ElasticNet with logistic loss (dashed blue) and sparse SVM (solid red). \label{fig:LRkfix2}} 
{Results correspond to average values obtained over $10$ data sets with $p=1,000$, $k_{true}=30$, $\rho=0.3$, $SNR \rightarrow \infty$ and increasing $n$ from $100$ to $1,300$.}
\end{figure*}

In terms of computational complexity, Figure \ref{fig:SVM.time} represents the number of cuts (left panel) and computational time (right panel) required by our cutting-plane algorithm as the problem size $n$ increases for a fixed value of $\gamma \times n$. For low values of $n$, the number of cuts is in the thousands. Surprisingly, computational effort does not increase with $n$. On the contrary, having more observations reduces the number of cuts down to less than twenty. Computational time evolves similarly: the algorithm reaches the time limit (here, 3 minutes) when $n$ is low, but terminates in a few seconds for high values. For sparse linear regression, \citet{bertsimas2016cio} observed a similar, yet even sharper, phenomenon which they referred to as a phase transition in computational complexity.  The threshold value, however, increases with $\gamma$. In other words, when $n$ is fixed, computational time increases as $\gamma$ increases, which corroborates the intuition that in the limit $\gamma \rightarrow 0$, $w^\star = 0$ is obviously optimal, while the problem can be ill-posed as $\gamma \rightarrow +\infty$. Since the right regularization parameter is unknown \textit{a priori}, one needs to test high but sometimes relevant values of $\gamma$ for which the time limit is reached and the overall procedures terminates in minutes, while \verb|glmnet| requires less than a second. As for the choice of the time limit, it does not significantly impact the performance of the algorithm: As shown on Figure \ref{fig:LRprofile}, the algorithm quickly finds the optimal solution and much of the computational time is spent improving on the lower bound, i.e., proving the solution is indeed optimal. {  We will further explore the numerical scalability of the outer-approximation algorithm in Section \ref{ssec:exp.scale}.}
\begin{figure*}[h]
\FIGURE
{\includegraphics[width=.45\linewidth]{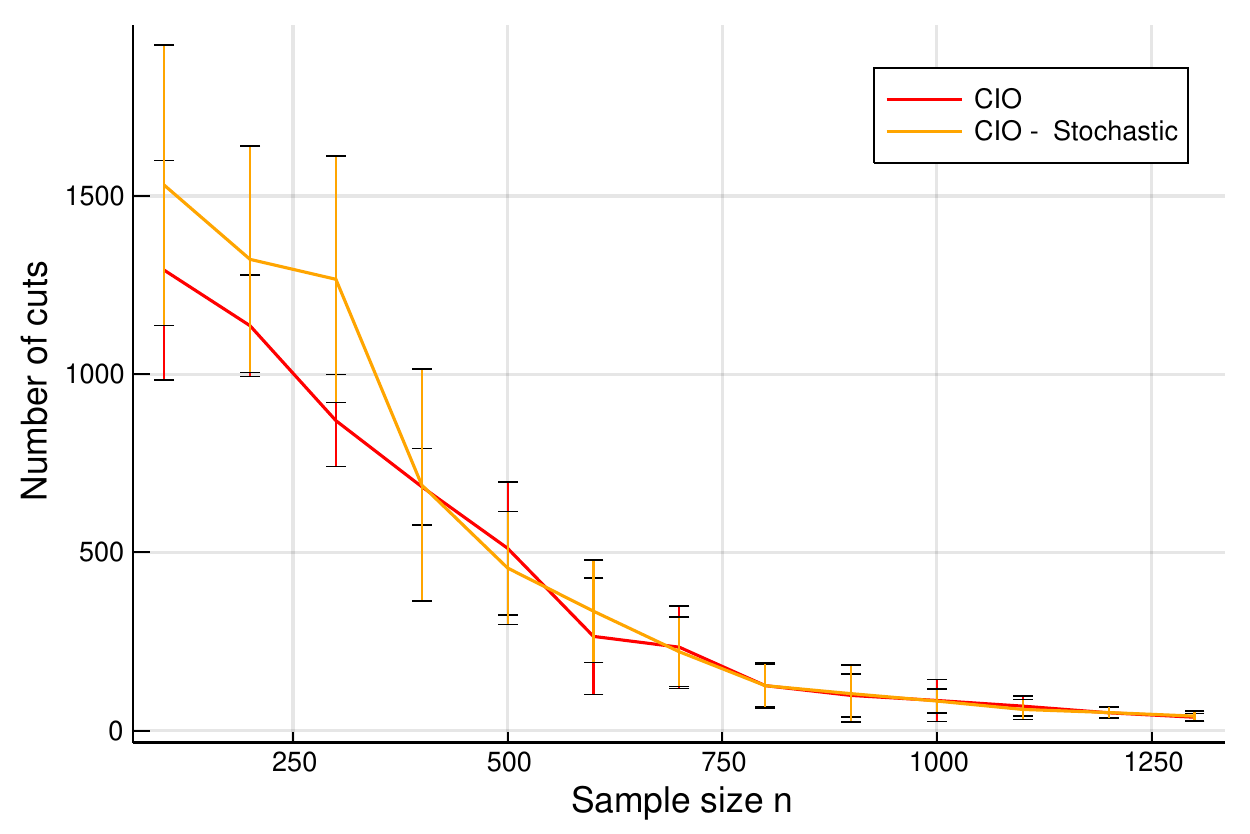}
\includegraphics[width=.45\linewidth]{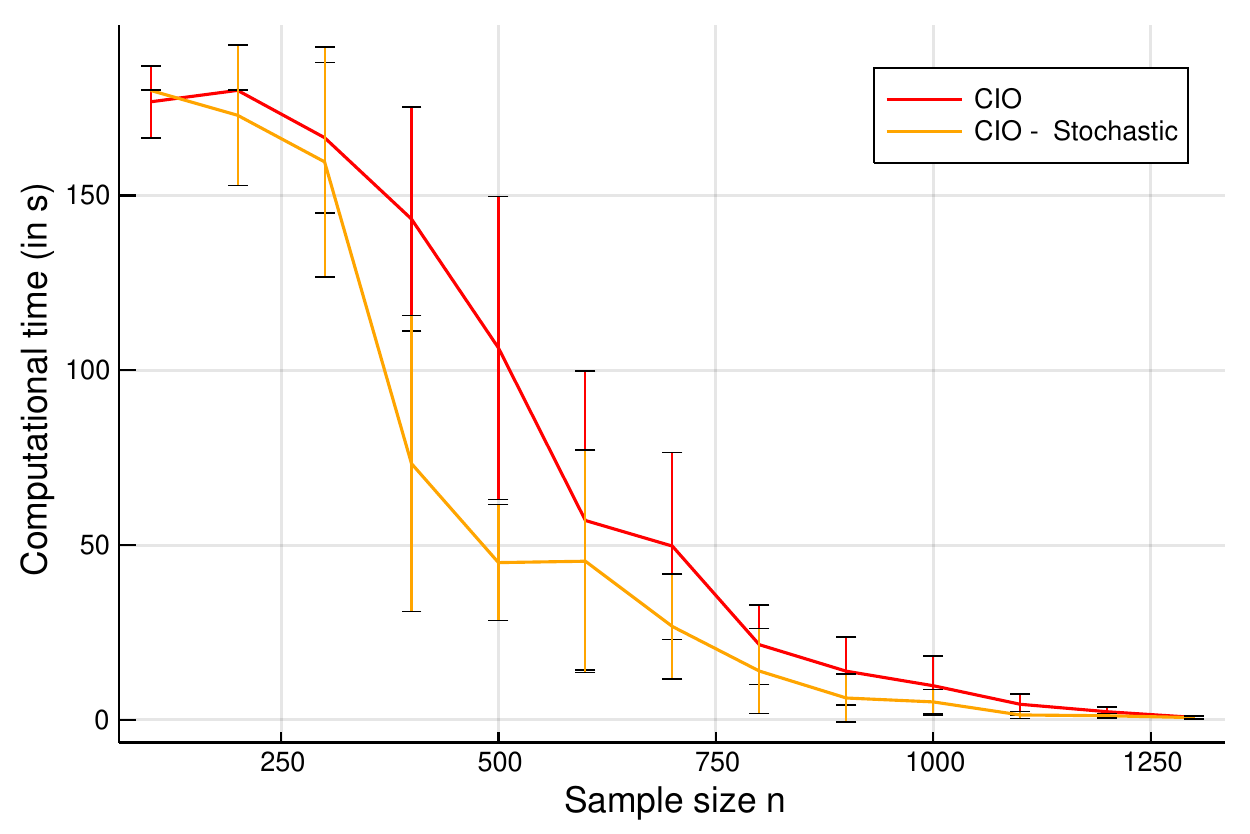} \label{fig:SVM.time}}
{Evolution of the number of cuts (left panel) and computational time (right panel) required by the outer-approximation algorithm with Hinge loss as sample size $n$ increases}
{Results correspond to average values obtained over $10$ data sets with $p=1,000$, $k_{true}=30$, $\rho=0.3$, $SNR \rightarrow \infty$, $\gamma = \gamma_0 / n$ with $\gamma_0 = 2^5 p / (k \max_i \|x_i \|^2)$ and increasing $n$ from $100$ to $1,300$. Algorithm \ref{OA} is initialized with the Lasso solution.}
\end{figure*}
\begin{figure*}[h]
\FIGURE
{\includegraphics[width=.6\linewidth]{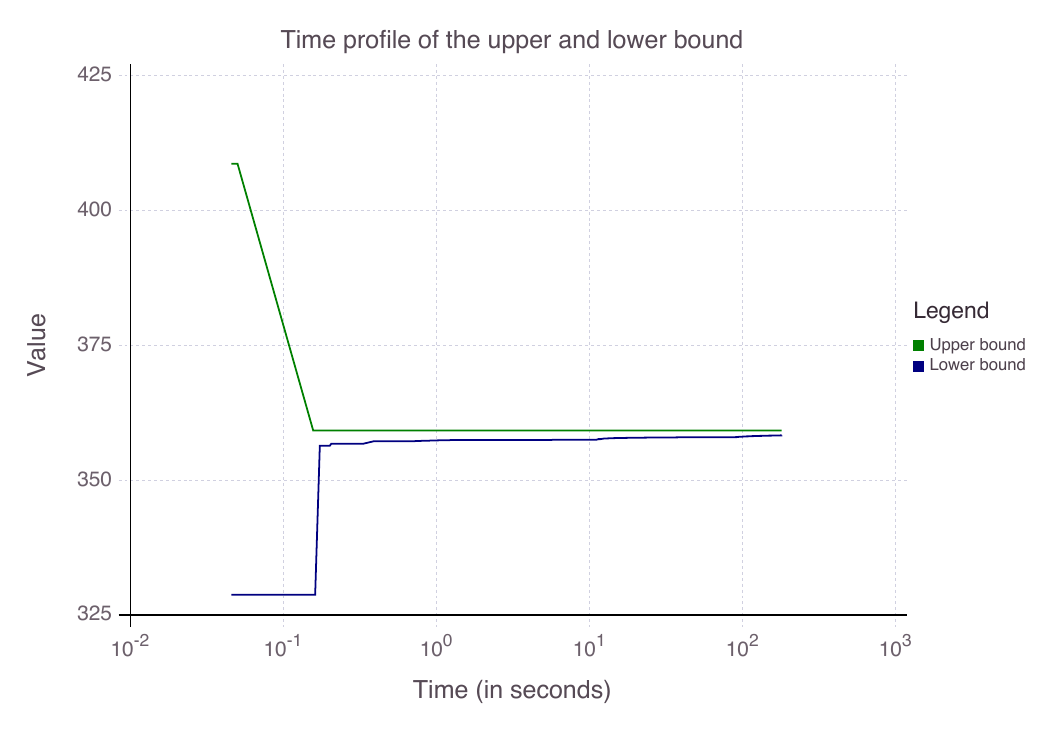}}
{  Evolution of the upper (best feasible solution, in green) and lower bounds (in blue) in Algorithm \ref{OA} as computational of time (in log scale) increases. \label{fig:LRprofile}}
{Results for one problem instance with $p=1,000$, $k_{true}=30$, $\rho=0.3$, $SNR \rightarrow \infty$, $\gamma = \gamma_0 / n$ with $\gamma_0 = 2^7 p / (k \max_i \|x_i \|^2)$ and  $n=600$.}
\end{figure*}
\subsubsection{...and nothing but the truth}
In practice, however, the length of the true support $k_{true}$ is unknown \textit{a priori} and is to be determined using cross-validation. Given a data set with a fixed number of samples $n$ and features $p$, we compute classifiers with different values of sparsity parameter $k$ and choose the value which leads to the best accuracy on a validation set. Irrespective of the method, AUC as a function of sparsity $k$ should have an inverted-U shape: if $k$ is too small, not enough features are taken into account to provide accurate predictions. If $k$ is too big, the model is too complex and overfits the training data. Hence, there is some optimal value $k^\star$ which maximizes validation AUC (equivalently, one could use misclassification rate instead of AUC as a performance metric). Figure \ref{fig:LRkcv1} represents the evolution of the AUC 
on a validation set as sparsity $k$ increases for Lasso and the exact sparse logistic regression. The exact CIO formulation leads to an optimal sparsity value $k^\star_{CIO}$ which is much closer to the truth than $k^\star_{Lasso}$, and this observation remains valid when $n$ increases as shown on the left panel of Figure \ref{fig:LRkcv2}. In addition, Figure \ref{fig:LRkcv2} also exposes a major deficiency of Lasso as a feature selection method: even when the number of samples increases, Lasso fails to select the relevant features \emph{only} and returns a support $k^\star_{Lasso}$ much larger than the truth whereas $k^\star_{CIO}$ converges to $k_{true}$ quickly as $n$ increases, hence selecting the truth and nothing but the truth.
\begin{figure*}[h]
\FIGURE
{\includegraphics[width=.5\linewidth]{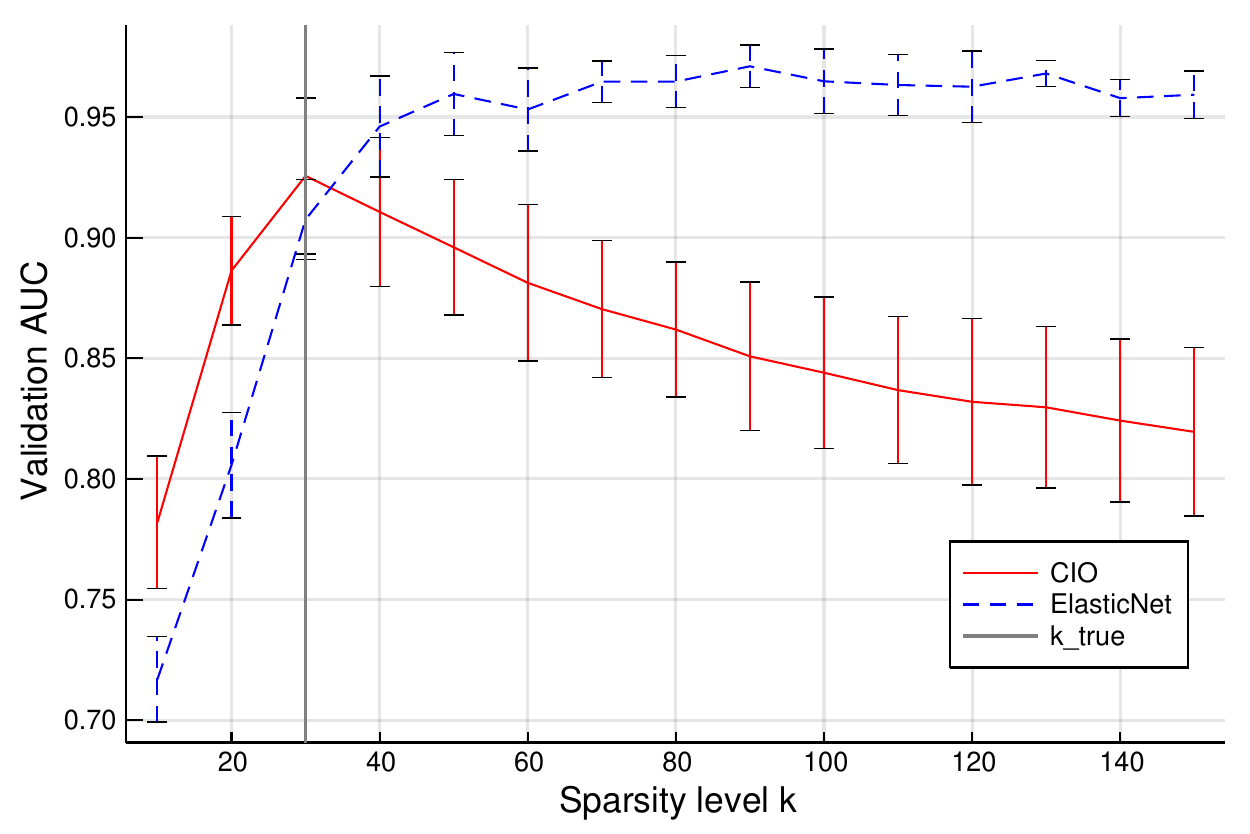}}
{Evolution of validation AUC as sparsity of the classifier $k$ increases, for ElasticNet (dashed blue) and sparse SVM (solid red). \label{fig:LRkcv1}}
{Results correspond to average values obtained over $10$ data sets with $n=700$, $p=1,000$, $k_{true}=30$ (black vertical line), $\rho=0.3$, $SNR \rightarrow \infty$ and increasing $k$ from $0$ to $100$.}
\end{figure*}
\begin{figure*}[h]
\FIGURE
{\includegraphics[width=.45\linewidth]{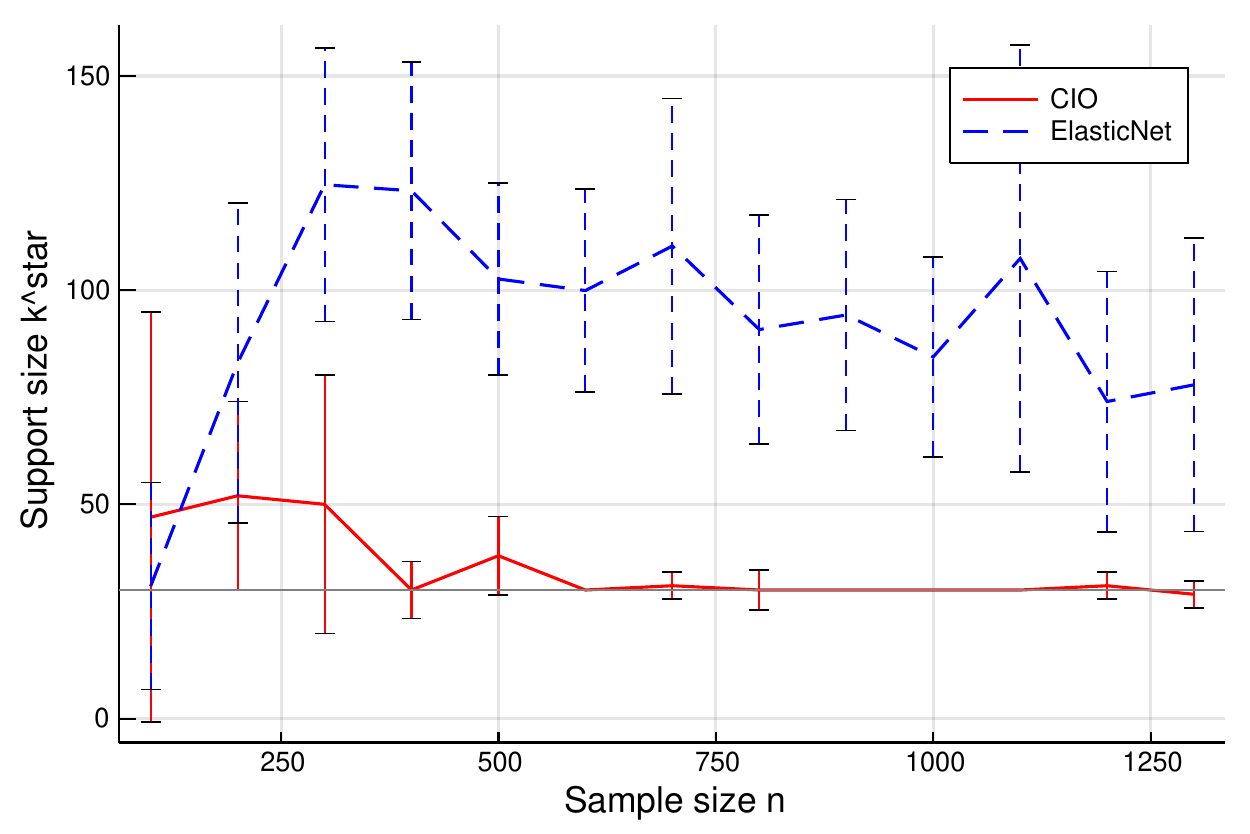}
\includegraphics[width=.45\linewidth]{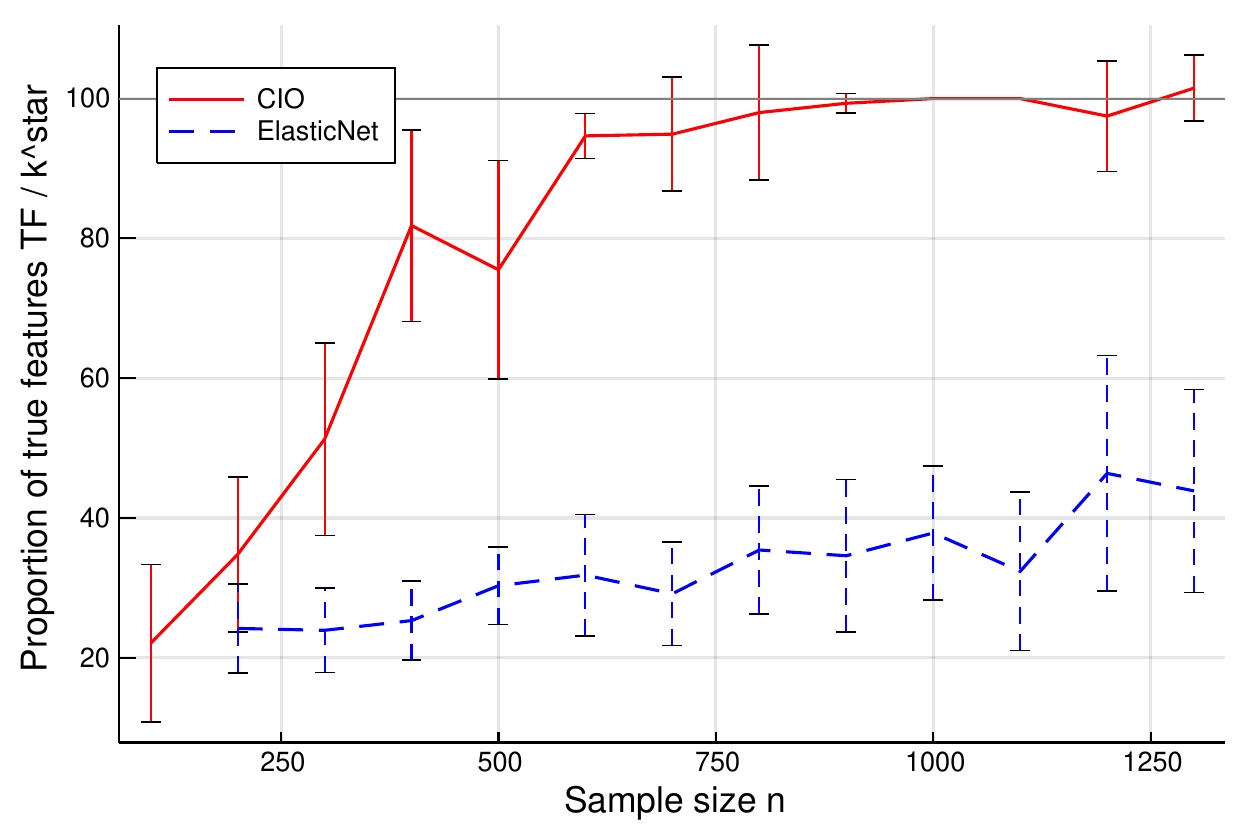}}
{Evolution of optimal sparsity $k^\star$ (left) and accuracy rate $A / k^\star$ (right) as sample size $n$ increases, for ElasticNet (dashed blue) and sparse SVM (solid red). \label{fig:LRkcv2}} 
{Results correspond to average values obtained over $10$ data sets with $p=1,000$, $k_{true}=30$, $\rho=0.3$, $SNR \rightarrow \infty$ and increasing $n$ from $100$ to $1,900$.}
\end{figure*}

{  Our findings are consistent with other comparisons of $\ell_0$ constrained estimators with Lasso-type regularizations. For instance, we refer to \citet{bertsimas2019sparse} for more extensive numerical comparisons, in regression and classification setting, under various regimes of noise and correlation.} 

\subsection{  Scalability on synthetic data} \label{ssec:exp.scale}
{  We use the opportunity of synthetic data to numerically assess how computational time is impacted by the size of the problem - as measured by the number of samples, $n$, and features, $p$ - and the 2 hyper parameters $\gamma$ and $k$. } 

{\bf Impact of sample size $n$: } As previously observed, computational time does not increase as $n$ increases. On the contrary, for a fixed value of $\gamma \times n$, computational time decreases as $n$ grows. As displayed on Figure \ref{fig:SVM.time} for Hinge loss, this phenomenon is observed for both the standard and stochastic version of the cutting-plane algorithm. Observe that the stochatic cutting plane slightly increases the number of cuts for low values of $n$, yet, since generating a cut is less expensive, computational time is reduced by a factor 2 for the hardest instances.  
On this regard, the stochastic version of the cutting-plane algorithm is more beneficial when using the logistic loss compared with the Hinge loss, as depicted on Figure \ref{fig:LogReg.time}. Indeed, number of cuts are reduced by a factor 5 and computational time by a factor 10. In our understanding, this substantial improvement in the logistic case is due to the fact that the relationship $\alpha_i \in \partial \ell(y_i, x_i^\top w)$ uniquely defines $\alpha$ and hence leads to higher quality cuts than in the case of the Hinge loss where $\ell$ is not differentiable at 0.
\begin{figure*}[h]
\FIGURE
{\includegraphics[width=.45\linewidth]{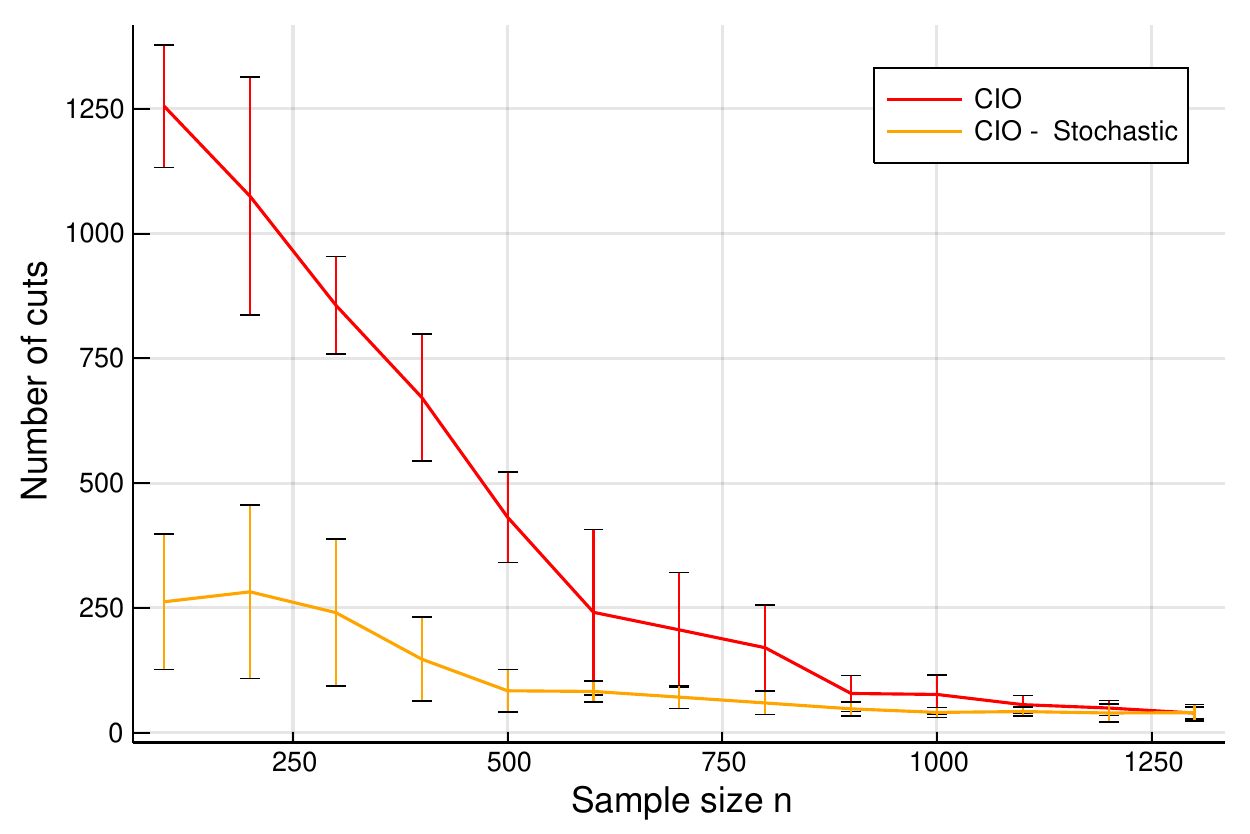}
\includegraphics[width=.45\linewidth]{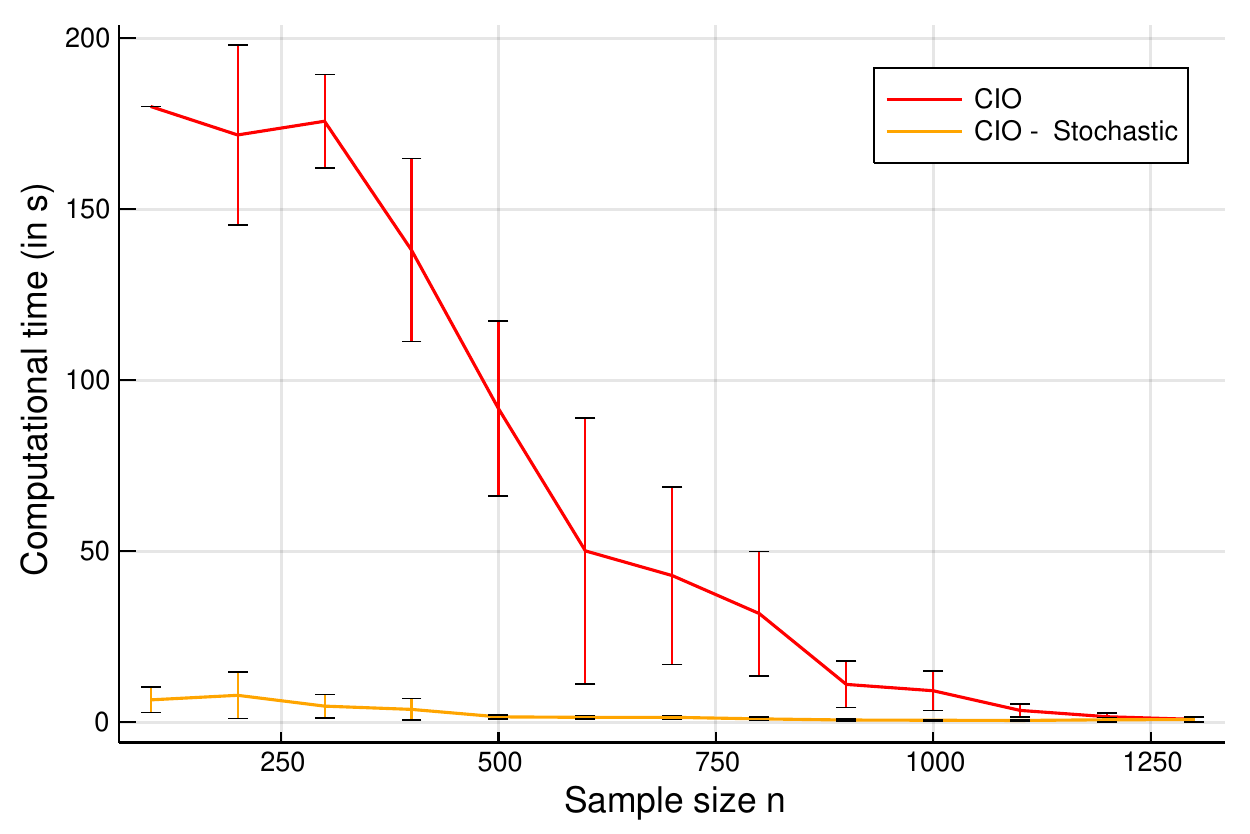} \label{fig:LogReg.time}}
{Evolution of the number of cuts (left panel) and computational time (right panel) required by the outer-approximation algorithm with Logistic loss as sample size $n$ increases.}
{Results correspond to average values obtained over $10$ data sets with $p=1,000$, $k_{true}=30$, $\rho=0.3$, $SNR \rightarrow \infty$, $\gamma = \gamma_0 / n$ with $\gamma_0 = 2^5 p / (k \max_i \|x_i \|^2)$ and increasing $n$ from $100$ to $1,300$. Algorithm \ref{OA} is initialized with the Lasso solution.}
\end{figure*}

{\bf Impact of regularization $\gamma$: } As $\gamma$ increases, however, computational time and the number of cuts required sharply increases, as depicted in Figure \ref{fig:time.gamma}. This phenomenon is consistently observed for all problem sizes and loss functions. Since the proper value of $\gamma$ is unknown a priori and needs to be cross-validated, this could result in prohibitive overall computational time. However, two aspects need to be kept in mind: First, computational time, here, corresponds to time needed to certify optimality, although high-quality solutions - or even the optimal solution - can be found much faster. Second, we do not implement any warm-starting strategy between instances with different values of $\gamma$ within the grid search. Smarter grid search procedures \citep{kenney2018efficient} could further accelerate computations. 
\begin{figure*}[h]
\FIGURE
{\includegraphics[width=.45\linewidth]{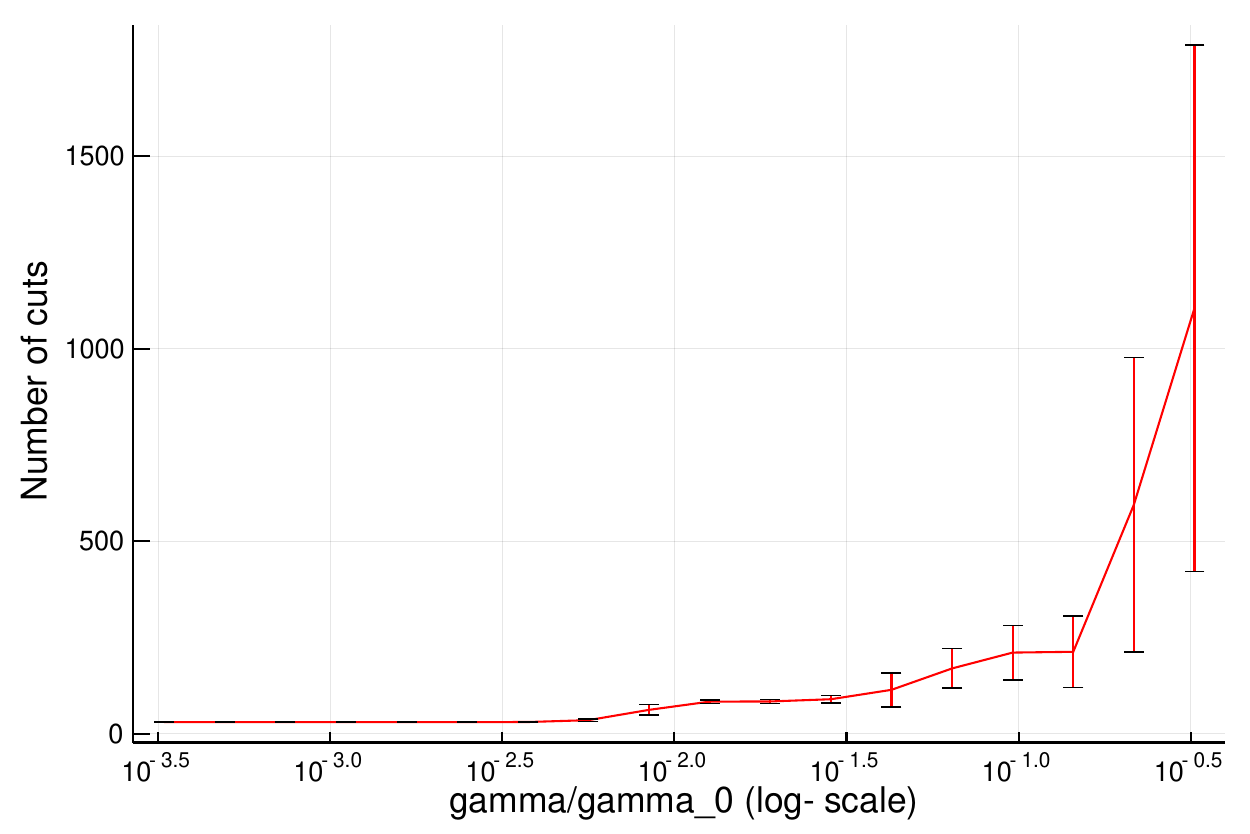}
\includegraphics[width=.45\linewidth]{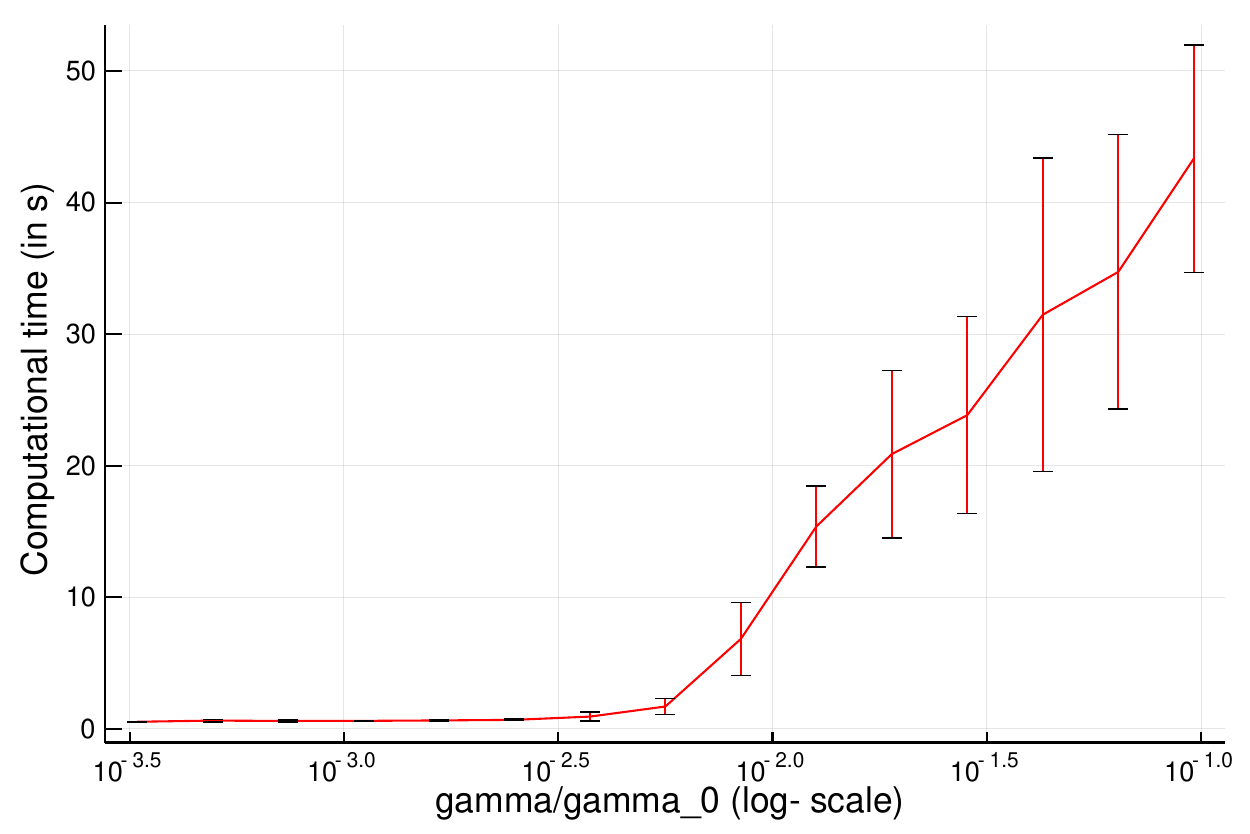}}
{Evolution of the number of cuts (left panel) and computational time (right panel) required by the outer-approximation algorithm with Hinge loss as sample size $\gamma$ increases. \label{fig:time.gamma}}
{Results correspond to average values obtained over $5$ data sets with $p=10,000$, $n=1,000$, $k=k_{true}=5$, $\rho=0$, $SNR =10$. The x-axis represents $\gamma / \gamma_0$ with $\gamma_0 = 10/n$. Algorithm \ref{OA} is initialized with the Lasso solution.}
\end{figure*}

{\bf Impact of feature size $p$ and sparsity level $k$: } The master problem \eqref{eqn:cio} obviously depends on the number of features $p$ and the sparsity level $k$ through the feasible set $\mathcal{S}_k^p$. As a first-order approximation, the combinatorial complexity of the problem can be captured by the total number of potential supports of size $k$, $\binom{p}{k}$. So one could expect a linear dependency on $p$ for $k$ fixed, and an exponential dependency on $k$, $p$ being fixed. This intuition is largely confirmed empirically, as reported in Table \ref{tab:time.p}: In low sparsity regimes ($k=5$), we can solve instances with up to $50,000$ features under a minute, while denser problems are not solved to optimality after 30 minutes even for $p=5,000$. Our experiments also highlight the benefits from having a good initialization solution $s_1$ and using the stochastic cut generation technique. As displayed in Table \ref{tab:time.p}, together, these enhancements reduce computational time up to a factor of $10$.
\begin{table}
\TABLE
{Computational time (in seconds) of Algorithm \ref{OA} for large $p$ and varying $k$. \label{tab:time.p}}
{\begin{tabular}{lccccc}
 \hfill $k=5$, $p=$ & $10,000$ & $20,000$ & $30,000$ & $40,0000$ & $50,000$ \\ \toprule
 Hinge loss (glmnet) & 18.2 & 29.4 & 29.3 & 30.7 & 148.7 \\
 Hinge loss (SubsetSelection) & 15.2 & 19.8 & 42.2 & 43.3 & 40.8 \\
 Hinge loss (SubsetSelection) Stochastic & 9.1 & 28.5 & 31.0 & 42.9 & 55.1 \\
 \\
  \hfill $p=5,000$, $k=$ & $1$ & $5$ & $10$ & $15$ &  $20$ \\ \toprule
Hinge loss (glmnet) & 0.6 & 7.6 & 569.3 & 702.6 & 1800 (40\%) \\
Hinge loss (SubsetSelection) & 0.2 & 7.5 & 203.0 & 537.0 & 1800 (46\%) \\
Hinge loss (SubsetSelection) Stochastic & 0.3 & 7.6 & 65.8  & 425.6 & 1800 (41\%) \\
\bottomrule
\end{tabular}
}{Instances are generated with $k=k_{true}$, $SNR = 10$, and $\gamma = 10/n$ as in \citep{dedieu2020learning}. We indicate in parentheses the warm-start method used. SubsetSelection refers to the Boolean relaxation of \eqref{eqn:sparse_reg_class} as implemented in \citet{bertsimas2019sparse}. If the algorithm did not converge within the given time budget (1800 seconds), we report optimality gap in parenthesis.}
\end{table}
\subsection{Experiments on real-world data sets}
We now {  illustrate the practical implications of sparse classification algorithms } on real-world data, of various size and dimensions. 
\subsubsection{Over-determined regime $n>p$} 
We consider data sets from the UC Irvine Machine Learning Repository (available at \url{https://archive.ics.uci.edu/ml/datasets.html}), split them into a training and a test set ($80\% / 20\%$), calibrate a classifier on the training data, using cross-validation to fit the hyper-parameters $k$ and $\gamma$ in the sparse case and $\lambda$ and $\alpha$ in the ElasticNet (Enet) case, and compare AUC on the test set for both methods. Characteristics of these data sets and experimental results are given in Table \ref{tab:LRuci_data}. Experiments clearly demonstrate that (a) our outer-approximation algorithm scales to data sets of practical relevance, (b) cardinality constrained formulations generally lead to sparser classifiers than ElasticNet (on 7 out 8 data sets) with comparable predictive power, with ElasticNet being more accurate on 3 out of 8 data sets. It suggests that features selected by our discrete optimization carry more relevant information that those obtained by ElasticNet. Yet, since these data sets contain a limited number of features $p$, they may not make a strong case for exact sparse classification methods compared to Lasso, the original problem being relatively sparse already. Therefore, we investigate the under-determined regime $n < p$ in the next section.

\begin{table}
\TABLE
{Comparative results of Lasso and sparse logistic regression on data sets from UCI ML Repository. \label{tab:LRuci_data}}
{
\begin{tabular}{lllcc|cc|cc}
& & & \multicolumn{2}{c}{ElasticNet} &\multicolumn{2}{c}{CIO - Hinge} & \multicolumn{2}{c}{CIO - Logistic} \\
Data set & $n$ & $p$ & $k$ & AUC & $k$ & AUC & $k$ & AUC \\
\midrule
Banknote Authentication & $1,372$ & $5$ & 
$4.0$ & $1.000$ &
$3.7$ & $1.000$ &
$\bf 3.2$ & $1.000$ \\
Breast Cancer & $683$ & $10$ & 
$7.8$ & $\bf 0.991$ & 
$\bf 5.4$ & $0.990$ & 
$5.5$ & $0.989$ \\
Breast Cancer (Diagnostic) & $569$ & $31$ & 
$21.1$ & $\bf 0.998$ & 
$12.0$ & $0.997$ & 
$\bf 11.4$ & $0.995$ \\
Chess (Rook vs. Pawn)& $3,196$ & $38$ &
$13.8$ & $0.867$ & 
$15.6$ & $0.862$ & 
$\bf 12.6$ & $\bf 0.868$ \\
Cylinder bands & $277$ & $484$ &
$\bf 98$ & $0.701$ & 
$111.0$ & $\bf 0.728$ & 
$139.0$ & $0.719$ \\
Magic Telescope & $19,020$ & $11$ &
$9.2$ & $0.840$ & 
$6.4$ & $0.843$ & 
$\bf 6.2$ & $\bf 0.844$  \\
QSAR Biodegradation & $1,055$ & $42$ & 
$35.8$ & $0.924$ & 
$\bf 18.5$ & $\bf 0.950$ &
$21.8$ & $0.947$ \\
Spambase & $4,601$ & $58$ & 
$56.4$ & $\bf 0.960$ & 
$27.2$ & $0.957$ &
$\bf 21.6$ & $0.954$ \\
\bottomrule
\end{tabular}
}
{AUC are computed out-of-sample on a test set comprised of $20\%$ of the initial data. Results are averaged over 10 different splits into train/validation data for cross-validation of the hyper-parameters.}
\end{table}

\subsubsection{Under-determined regime $p>n$}
Performance of sparse classification in the under-determined regime is crucial for two reasons: Since the amount of data available is limited, such regime favors estimators which can efficiently recover the truth even when the sample size $n$ is small with regard to $p$. More importantly, under-determined regimes occur in highly impactful applications, such as medical research. To show the direct implications of our method on this field of research, we used data from The Cancer Genome Atlas Research Network (available at \url{http://cancergenome.nih.gov}) on $n = 1,145$ lung cancer patients. Actually, tumor types often have distinct subtypes, each of them having its own genetic signature. In our sample for instance, $594$ patients ($51.9\%$) suffered from Adenocarcinoma while the remaining $551$ patients ($48.1\%$) suffered from Squamous Cell Carcinoma. The data set consists of gene expression data for $p=14,858$ genes  for each patient. We apply both sparse and Lasso classification to identify the most relevant genes to discriminate between the two subtypes and compile the results in Table \ref{tab:LRcancer}. The first conclusion to be drawn from our results is that the exact sparse classification problem scales to problems of such size, which is far above data sets usually encountered in the gene selection academic literature. In addition, explicitly constraining sparsity of the classifier leads to much sparser, thus more interpretable results with little compromise on the predictive power: Sparse SVM reaches an AUC of $0.977$ with only $38$ features while the $\ell_1$-regularized classifier selects ten times more genes for a $+0.005$ gain in AUC.
\begin{table}[h]
\TABLE
{Comparative results of $\ell_1$ regularized and $\ell_0$ constrained estimators on the Lung Cancer data. \label{tab:LRcancer}}
{
\begin{tabular}{lllcc|cc|cc}
& & & \multicolumn{2}{c}{ElasticNet} &\multicolumn{2}{c}{CIO - Hinge} & \multicolumn{2}{c}{CIO - Logistic} \\
Data set & $n$ & $p$ & $k$ & AUC & $k$ & AUC & $k$ & AUC \\
\midrule
Lung cancer & $1,145$ & $14,858$ & 
$378.4$ & $\bf 0.982$ &
$\bf 38$ & $0.977$ &
$90$ & $0.980$ 
\\
\bottomrule
\end{tabular}
}
{AUC are computed out-of-sample on a test set comprised of $20\%$ of the initial data. Results are averaged over 10 different splits into train/validation data for cross-validation of the hyper-parameters.}
\end{table}

\section{Towards a theoretical understanding of asymptotic support recovery}
\label{sec:theory}
As mentioned in the introduction, a large body of literature has provided information-theoretic limitations \citep{wainwright2009information,wang2010information,gamarnik17} or theoretical guarantees of support recovery by some specific algorithms \citep{wainwright2009sharp,pilanci2015sparse} for sparse linear regression, {  which supported the empirical observation of a phase transition \citep{donoho2006breakdown,bertsimas2016cio}. In classification, however, we did not  observe such a sharp phenomenon. In this section, we provide some intuition on the theoretical mechanisms involved for perfect support recovery in classification specifically. We prove an information-theoretic sufficient condition to achieve perfect support recovery and compare it with analogous bounds for linear regression, as well recent theoretical results on 1-bit compressed sensing \citep{jacques2013robust,scarlett2017limits}.}
\subsection{Notations and assumptions} 
To simplify the analysis, we consider a stylized framework where the data is generated according to the equation 
\begin{align*}
y_i = sign\left(x_i^\top w^{\star} + \varepsilon_i\right),    
\end{align*}
where $x_i$ are i.i.d. standard Gaussian random variables, $\varepsilon_i \sim \mathcal{N}(0,\sigma^2)$, $w^{\star} \in \mathbb{R}^p$ with $\|w^{\star}\|_0=k$ and $sign(\lambda) = 1$ if $\lambda >0$, $-1$, otherwise. Given a classifier $w$ predictions will be made according to the rule 
\begin{align*}
\hat y_i(w) = sign\left(x_i^\top w \right).
\end{align*}
It is obvious from the definition that for any $w \in \mathbb{R}^p$, $sign\left( \lambda x^\top w \right) = sign\left(x^\top w \right), \forall \lambda >0$. In other words, predictions made by a classifier are insensitive to scaling. As a consequence, the difference in prediction between two classifiers should demonstrate the same invariance and indeed only depends on the angle between the classifiers as formally stated in Lemma \ref{orthant} (proof in Appendix \ref{sec:theory_proof}) . This observation does not hold for $sign\left(x^\top w + \varepsilon \right)$, because of the presence of noise.
\begin{lemma}
\label{orthant}
Assume $x \sim \mathcal{N}(0,1)$ and $\varepsilon \sim \mathcal{N}(0,\sigma^2)$ are independent. Then, for any $w,w' \in \mathbb{R}^p$ we have that 
\begin{eqnarray}
\mathbb{P}\left( sign\left(x^\top w \right) \neq sign\left(x^\top w'+ \varepsilon \right) \right) &= \dfrac{1}{\pi} \arccos \left( \dfrac{w^\top w'}{\|w\|\sqrt{\|w'\|^2 + \sigma^2}} \right). \label{eq:lema}
\end{eqnarray}
\end{lemma}
We consider classifiers with binary entries $w^\star \in \{0,1\}^p$ only, similar to the work of \citet{gamarnik17} on sparse binary regression. Moreover, we learn the optimal classifier from the data by solving the minimization problem
\begin{equation}
\min_{w \in \{0,1\}^p} \dfrac{1}{n} \sum_{i=1}^n \textbf{1}\left(\hat y_i(w) \neq y_i \right) \mbox{ s.t. } \| w \|_0 = k,
\label{eqn:phi}
\end{equation}
where the loss function above corresponds to the empirical misclassification rate. Even though it is not a tractable loss function choice in practice, it demonstrates some interesting theoretical properties: it isolates the probabilistic model used to generate the data from the behavior of the optimal value. Indeed, for any classifier $w$, the empirical misclassification rate $\sum_{i=1}^n \textbf{1}\left(\hat y_i(w) \neq y_i \right)$ follows a Binomial distribution, as the covariate data are independent. In addition, the problem \eqref{eqn:phi} can be considered as the authentic formulation for binary classification, while other loss functions used in practice such as Hinge and logistic loss are only smooth proxies for the misclassification rate, used for their tractability and statistical consistency \citep{steinwart2002support,zhang2004statistical}. 

\subsection{Intuition and statement on sufficient conditions}
For a given binary classifier {  $w \in \{0,1\}^p$ of sparsity $k$}, the accuracy of the classifier (the number of true features it selects) is equal to the inner product of $w$ with $w^\star$:
\begin{align*}
A(w) = |\lbrace j : w_j \neq 0, w^\star_j \neq 0 \rbrace| = |\lbrace j : w_j =1, w^\star_j = 1 \rbrace| = \textstyle\sum_j w_j w^\star_j = w^\top w^\star.
\end{align*}
Consider a binary sparse classifier $w$, i.e., $\Vert w \Vert_0=k$, with accuracy $w^\top w^\star = \ell$. Then, it follows that the indicators $\textbf{1}\left(\hat y_i(w) \neq y_i \right)$ are distributed as independent Bernoulli random variable sharing the success parameter
\begin{align*}
  q (\ell;k,\sigma^2) & := \mathbb{P}\left( sign\left(x_i^\top w \right) \neq sign\left(x_i^\top w^\star+ \varepsilon_i \right) \right),\\
                      & \displaystyle = \dfrac{1}{\pi} \text{arccos}\left( \dfrac{\ell}{\sqrt{k(k+\sigma^2)}} \right).
\end{align*}
The success parameter $q_\ell=q (\ell;k,\sigma^2)$  can be checked to be a decreasing concave function of $\ell$. That is, the more accurate our binary classifier $w$, the smaller the probability of {  misclassification}. The previous should come as no surprise to anybody. The central limit theorem states that 
$$\sqrt{n} \left[ \dfrac{1}{n} \sum_{i=1}^n \textbf{1}\left(\hat y_i(w) \neq y_i \right) - q_l \right] \rightarrow \mathcal{N}\left(0, {q_\ell(1-q_\ell)}\right),$$ as $n \rightarrow \infty$. In words, asymptotically in $n$, a given classifier $w$ will have an empirical misclassification rate close to $q_\ell$. Since $q_\ell$ is decreasing in $\ell$, the truth $w^\star$ for which $\ell=k$ should minimize the misclassification error among all possible supports. As observed empirically, the number of true features selected corresponds to the true sparsity when $n$ is sufficiently large (see Figures \ref{fig:LRkfix1}). Intuitively, $n$ should be high enough such that the variance on the performance of each support $\tfrac{q_\ell(1-q_\ell)}{n}$ is small, taken into account that there are $\binom{k}{\ell} \binom{p-k}{k-\ell}$ possible supports with exactly $\ell$ correct features. In this case, it should be rather unlikely that the binary classifier with the smallest empirical misclassification rate is anything other than the ground truth $w^\star$. We will now make the previous intuitive argument more rigorous.

Because we aim at minimizing the misclassification rate, we are guaranteed to recover the true support $w^\star$ if there exists no other support $w$ with an empirical performance at least as good as the truth.
\begin{theorem}
\label{sufficient}
 
We assume the data is generated according to the equation $ y_i = sign\left(x_i^\top w^{\star} + \varepsilon_i\right)$, 
where $x_i$ are i.i.d. standard random variables, $\varepsilon_i \sim \mathcal{N}(0,\sigma^2)$, $w^{\star} \in \{0,1\}^p$ with $\|w^{\star}\|_0=k$. Given a classifier $w$, we denote $\hat y_i(w) := sign\left(x_i^\top w \right)$.
For any two binary classifiers $w_1$ and $w_2$, let $\Delta(w_1,w_2)$ denote the difference between their empirical performance, i.e.,
\begin{align*}
\Delta(w_1,w_2) := \dfrac{1}{n}\sum_{i=1}^n \textbf{1}\left(\hat y_i(w_1) \neq y_i \right) - \textbf{1}\left( \hat y_i(w_2) \neq y_i \right).
\end{align*}
Assume that $p \geqslant 2k$. Then there exist a threshold $n_0 >0$ such that for any $n > n_0$, 
\begin{align*}
  \mathbb{P}\left(\exists w \neq w^\star, \, w\in \{0,1\}^p, \, \|w\|_0=k \mbox{ s.t. }\Delta(w,w^\star) \leqslant 0 \right) \leqslant e^{-\frac{n-n_0}{2\pi^2 k (\sigma^2+2)}}.
\end{align*}
Moreover, we have $n_0 < C \left(2 + \sigma^2\right) k \log(p-k)$ for some absolute constant $C > 0$.  \\

In other words, if $n \geqslant C \left(2 + \sigma^2\right) k \log(p-k) + 2 \pi k (\sigma^2 +2) \log(1/\delta)$ for some $\delta \in (0,1)$, then   $$\mathbb{P}\left(\exists w \neq w^\star, \, w\in \{0,1\}^p, \, \|w\|_0=k \mbox{ s.t. }\Delta(w,w^\star) \leqslant 0 \right) \leqslant \delta.$$
\end{theorem} 
The proof of Theorem \ref{sufficient} is given in Appendix \ref{sec:theory_proof}. From a high-level perspective, our sufficient condition on $n$ relies on two ingredients:  (a) the union bound, which accounts for the log-dependence in $p-k$ and which is found in similar results for regression and signal processing \citep{wainwright2009information,wang2010information} and (b)  controlling  of the individual probability $\mathbb{P}\left(\Delta(w,w^\star) \leqslant 0 \right)$ using large deviation bounds, which depends on the size of $w$ and $w^\star$, $k$, and the noise  $\sigma^2$.

Before comparing the claim of Theorem \ref{sufficient} with similar results from statistics and signal processing, let us remember that $k$, the sparsity level of the true classifier $w^\star$, is assumed to be known. To put this assumption in perspective with our previous simulations, our statement only concerns the best achievable accuracy when $k$ is fixed. 
\subsection{Discussion}
For regression, \citet{gamarnik17} proved that support recovery was possible from an information-theoretic point of view if 
\begin{align*}
n > n^\star = \dfrac{2k \log p}{\log\left(\dfrac{2k}{\sigma^2}+1\right)}.
\end{align*}
Note that our threshold $n_0$ for classification does not vanish in the low noise $\sigma^2$ setting. This observation is not surprising: the output $y_i$ depending only on the sign of $w^\top x_i$ can be considered as inherently noisy. An observation already made by \citet{scarlett2017limits}.

As mentioned earlier, recent works in 1-bit compressed sensing have developed algorithms to recover sparse classifiers which provably recover the truth as the number of samples increases \citep{gupta2010sample,plan2013robust,plan2013one}. In particular, \citet{plan2013robust} formulate the problem of learning $w$ from the observations as a convex problem and establish bounds on the $\ell_2$ error $\| w-w^\star\|_2$, in the case of logistic regression. In particular, they show that $n > C' \, k \log(2p/k)$ is sufficient to achieve an $\ell_2$ reconstruction error which is bounded. {  In contrast with our result, they exhibited an algorithm able to achieve low error in terms of $\ell_2$ distance with fewer samples than what we proved to be an information-theoretic sufficient condition. Yet, this  does not trivialize  our result. $\ell_2$ consistency is a related but distinct criterion for support recovery. Even a good estimate of $w^\star$ in terms of $\ell_2$ distance might have a very different support. The fact that their rate for $\ell_2$ consistency is faster than the rate for support recovery suggests that achieving a good $\ell_2$ error is presumably an easier problem. Similar observations were  made in the case of linear regression in \citet{wainwright2009information} and intuitively explained: given a good support recovery procedure, one can restrict the number of features, use standard methods to estimate the values of the $w_j$'s and hope to achieve a good $\ell_2$ error, while the reverse might not be true.}

Finally, \citet{scarlett2017limits} proved similar sufficient conditions for support recovery in 1-bit compressed sensing and accompanied them with necessary conditions and constants $C >0$ as tight as possible. To that extent, their result (Corollary 3) might appear stronger than ours. However, their condition is valid only for low sparsity and low signal-to-noise regimes. Theorem \ref{sufficient}, on the other hand, remains valid for all values of $k$ and $\sigma$. In particular, it holds even if $k$ scales linearly in $p$ and $\sigma$ is low: a regime for which \citet{scarlett2017limits} provide necessary (Corollary 4) but no sufficient conditions. More precisely, they prove that perfect support recovery cannot be achieved if the number of samples is below a threshold scaling as $p \sqrt{\log p}$, while our bound scales as $p \log p$ in this regime. Combined together, there is  a $\sqrt{\log p}$ factor between necessary and sufficient conditions, which hints at the absence of a clear phase transition in this setting. As illustrated in Figure \ref{fig:theory.summary}, there is an intermediate sample size regime where support recovery is neither provably impossible nor achievable. Of course, this regime could be a deficiency of the proof techniques used, but in any case, we believe it constitutes an exciting direction for future research.

\section{Conclusion}
\label{sec:conclusion}
In this paper, we have proposed a tractable binary convex optimization algorithm for solving sparse classification. Though theoretically NP-hard, our algorithm scales for logistic regression and SVM  in problems  with $n,p$ in $10,000$s within minutes. {  We also introduce a stochastic version of the cut generation process which further reduces computational time by a factor 2 to 10 on numerical experiments.} Comparing our method and Lasso-based estimates, we observe empirically that  as $n$ increases, the number of true features selected by both methods  converges to the true sparsity. We support our observations with information-theoretic sufficient conditions, stating that support recovery is achievable as soon as $n > n_0$, with $n_0 < C \left(2 + \sigma^2\right) k \log(p-k)$ for some positive constant $C$. This sufficient information-theoretic condition echoes and complements existing results in the literature for 1-bit compressed sensing. Apart from accuracy, the exact sparse formulation has an edge over Lasso  in the number of false features: as $n$ increases, the number of false features selected by our method converges to zero, while this  is not observed for Lasso. This phenomenon is also observed for  classifying the type of  cancer using gene expression data from the Cancer Genome Atlas Research Network  with $n=1,145$  lung cancer patients and  $p=14,858$ genes.
Sparse classification using logistic and hinge loss returns a classifier based on $90$ and $38$ genes respectively compared with $378.4$ genes for ElasticNet  with  similar predictive accuracy. 


\newpage
\begin{APPENDICES}
\section{Proof of the sufficient condition for a support recovery} \label{sec:theory_proof}

\subsection{Preliminary results on orthant probabilities}
Let us recall an analytical expression for the probability that a bivariate normal distribution assigns to the positive orthant.

\begin{lemma} \citep[p.\ 290]{cramer2016mathematical}
\label{biv_orthant}
Assume we are given a zero mean bivariate normal random variable $(n_1, n_2)$ with $\mathbb E[n_1^2] = \mathbb E[n_2^2]=1$ and covariance $\rho_{12} = \mathbb E[n_1 n_2]$. Then,
\[
  \mathbb{P} \left( n_1 \geqslant 0, n_2 \geqslant 0 \right) = \frac{1}{2 \pi} \left( \frac{\pi}{2} + \arcsin(\rho_{12}) \right).
\]
\end{lemma} 
By continuity of the density function of normal distributions, the probability of the positive orthant and its interior are equivalent. Lemma \ref{orthant} is an almost direct consequence of the previous result. We give here its proof.
\vspace{1em}
\proof{Proof of Lemma \ref{orthant}}
We can separate the event of interest in two disjunct cases as
\begin{align*}
  & \mathbb{P}\left( sign\left(x^\top w \right) \neq sign\left(x^\top w' + \varepsilon  \right) \right)  \\
  & \qquad = \mathbb{P}\left( x^\top w \leqslant 0, x^\top w'+ \varepsilon  > 0 \right) + \mathbb{P}\left( x^\top w > 0, x^\top w' + \varepsilon  \leqslant 0 \right).
\end{align*}
Each term corresponds to the probability that a zero mean  bivariate normal variable $(x^\top w, x^\top w' +\varepsilon)$ realizes in an appropriate orthant. We define the random variables $n_1 := x^\top w/\norm{w}$ and $n_2 := (x^\top w' +\varepsilon)/\sqrt{\| w' \|^2 + \sigma^2}$ and obtain
\[
  \mathbb{P}\left( sign\left(x^\top w \right) \neq sign\left(x^\top w' + \varepsilon  \right) \right) = \mathbb P(n_1\leq 0, n_2>0) + \mathbb P(n_1>0, n_2\leq 0).
\]
We have that $\mathbb E[n_1^2]=1$, $\mathbb E[n_2^2]=1$ and $\rho_{12} =\mathbb E[n_1 n_2] = {w^\top w'}/{(\| w \| \sqrt{\| w' \|^2 + \sigma^2})}$.
Using the analytical expressions of such orthant probabilities for bivariate normal random variables given in Lemma \ref{biv_orthant}, we have hence
\(
\mathbb{P}\left( sign\left(x^\top w \right) \neq sign\left(x^\top w' + \varepsilon \right) \right) = \frac{1}{\pi} \left( \frac{\pi}{2} -  \arcsin(\rho_{12}) \right) = \frac{1}{\pi} \arccos(\rho_{12}).
\)
\hfill\halmos
\endproof

We will need a minor generalization of Lemma \ref{biv_orthant} to the three dimensional case in the proof of Theorem \ref{sufficient}. 
\begin{lemma} \citep[p. 290]{cramer2016mathematical}
\label{triv_orthant}
Assume we are given a zero mean trivariate normal random variable $(n_1, n_2, n_3)$ with $\mathbb E[n_1^2] = \mathbb E[n_2^2]= \mathbb E[n_3^2]=1$ among which we have covariances $\rho_{12} = \mathbb E[n_1 n_2]$, $\rho_{13} = \mathbb E[n_1 n_3]$ and $\rho_{23} = \mathbb E[n_2 n_3]$. Then,
\[
\mathbb{P} \left( n_1 \geqslant 0, n_2 \geqslant 0, n_3 \geqslant 0 \right) = \frac{1}{4 \pi} \left( \frac{\pi}{2} + \arcsin(\rho_{12}) + \arcsin(\rho_{13}) + \arcsin(\rho_{23}) \right).
\]
\end{lemma} 


\subsection{Comparative performance of a given support with the truth}
We first prove a large deviation bound for $\mathbb{P}\left( \Delta(w,w^\star) \leqslant 0\right)$ for any given binary classifier $w$, depending on the number of true features it selects. The following result can be derived using Hoeffdings inequality as illustrated in its proof.
\begin{lemma} \label{large_dev}
Let $w\in\{0,1\}^p$ be a binary classifier such that $\| w\|_0 =k$ and $w^\top w^\star = \ell \in \{0,\ldots,k\}$. Its misclassification rate with respect to the ground truth satisfies the exponential bound
\begin{align*}
\mathbb{P}\left( \Delta(w,w^\star) \leqslant 0\right) &\leqslant \exp\left( -n \frac{(k-\ell)^2}{2 \pi^2 (k (k+\sigma^2)-\ell^2)}\right).
\end{align*}
\end{lemma}
\proof{Proof of Lemma \ref{large_dev}}
Let us consider a binary classifier $w\in\{0,1\}^p$ with sparsity $\| w\|_0 =k$ and true features $w^\top w^\star = \ell$. We compare the empirical misclassification rate of $w$ with the performance of the true support $w^\star$. We take the misclassification rate with respect to the ground truth $w^\star$ as
\begin{align*}
\Delta(w,w^\star) = \dfrac{1}{n}\sum_{i=1}^n \textbf{1}\left(\hat y_i(w) \neq y_i \right) - \textbf{1}\left( \hat y_i(w^\star) \neq y_i \right) =: \dfrac{1}{n}\sum_{i=1}^n Z_i
\end{align*}
which is composed of the sum of independent random variables $Z_i$ taking values in $\{-1, 0, 1\}$ such that 

\begin{align*}
Z_i= \begin{cases} +1, &\mbox{ if } y_i  = \hat y_i(w^\star) \neq \hat y_i(w), \\
-1, &\mbox{ if } y_i  = \hat y_i(w) \neq \hat y_i(w^\star), \\
0, &\mbox{ otherwise}
\end{cases}
\end{align*}
Each random variable $y_i$, $\hat y_i(w^\star)$, $\hat y_i(w)$ is the sign of the normally distributed quantities $x_i^\top w^\star + \varepsilon_i$, $x_i^\top w^\star$ and $x_i^\top w$ respectively. Let us define three zero mean random variables $n_1 = (x_i^\top w + e_i)/\sqrt{k + \sigma^2}$, $n_2 = x_i^\top w / \sqrt{k}$ and $n_3 = x_i^\top w^\star/\sqrt{k}$. Their covariance structure is characterized as $\rho_{12}=\mathbb E[n_1 n_2] = k/\sqrt{k(k+\sigma^2)}$, $\rho_{23}=\mathbb E[n_2 n_3] = \ell/k$ and $\rho_{13}=\mathbb E[n_1 n_3] = \ell/\sqrt{k(k+\sigma^2)}$.
We can then express the probabilities of each value of $Z_i$ as tridimensional orthant probabilities for these three zero mean correlated normal random variables  and use the analytical expression given in Lemma \ref{triv_orthant}. We hence arrive at
\begin{align*}
  \mathbb{P} \left( Z_i=1 \right) & =\mathbb{P}\!\left( x_i^\top w^\star\!+\! \epsilon_i\geq 0 , x_i^\top w \leqslant 0, x_i^\top w^\star \geqslant 0 \right) \!+\! \mathbb{P}\!\left( x_i^\top w^\star \!+\! \epsilon_i\leq 0 , x_i^\top w \geqslant 0, x_i^\top w^\star \leqslant 0 \right) \\
  &= \mathbb{P} \left( n_1\geq 0, n_2\leq 0, n_3 \geqslant 0 \right) +  \mathbb{P} \left( n_1 \leqslant 0, n_2\geq 0, n_3\leq 0 \right) \\ 
                                  &= \dfrac{1}{2 \pi} \left[ \dfrac{\pi}{2} - \arcsin\left(\dfrac{\ell}{\sqrt{k(k+\sigma^2)}}\right) + \arcsin\left(\dfrac{k}{\sqrt{k(k+\sigma^2)}}\right) - \arcsin\left(\dfrac{\ell}{k}\right) \right], \\
  \intertext{and equivalently}
  \mathbb{P} \left( Z_i=-1 \right) & = \mathbb{P}\!\left( x_i^\top w^\star\!+\! \epsilon_i\geq 0 , x_i^\top w \geqslant 0, x_i^\top w^\star \leqslant 0 \right) \!+\! \mathbb{P}\!\left( x_i^\top w^\star \!+\! \epsilon_i\leq 0 , x_i^\top w \leqslant 0, x_i^\top w^\star \geqslant 0 \right) \\
                                  &= \mathbb{P} \left( n_1\geq 0, n_2\geq 0, n_3\leq 0 \right) +  \mathbb{P} \left( n_1\leq 0, n_2\leq 0, n_3\geq 0 \right) \\
&= \dfrac{1}{2 \pi} \left[ \dfrac{\pi}{2} + \arcsin\left(\dfrac{\ell}{\sqrt{k(k+\sigma^2)}}\right) - \arcsin\left(\dfrac{k}{\sqrt{k(k+\sigma^2)}}\right) - \arcsin\left(\dfrac{\ell}{k}\right) \right]. 
\end{align*}
Evidently, we can characterize the probability of $Z_i=0$ as $\mathbb P(Z_i=0) = 1 - \mathbb P(Z_i=1) - \mathbb P(Z_i=-1)$. The mean of $Z_i$ is now easily found as the expression
\begin{align*}
\mathbb{E}[Z_i] &= \dfrac{1}{\pi} \left[\arccos\left(\dfrac{\ell}{\sqrt{k(k+\sigma^2)}}\right) - \arccos\left(\dfrac{k}{\sqrt{k(k+\sigma^2)}}\right)\right].\\
\end{align*}
Concavity of the arccos function on the interval $[0, 1]$ enables us to state the gradient inequalities
\begin{align*}
\arccos\left(\dfrac{k}{\sqrt{k(k+\sigma^2)}}\right) &\leqslant \arccos\left(\dfrac{\ell}{\sqrt{k(k+\sigma^2)}}\right) + \arccos^\prime \left(\dfrac{\ell}{\sqrt{k(k+\sigma^2)}}\right) \dfrac{k-\ell}{\sqrt{k(k+\sigma^2)}} \\
&= \arccos\left(\dfrac{\ell}{\sqrt{k(k+\sigma^2)}}\right) - \sqrt{\dfrac{k(k+\sigma^2)}{k(k+\sigma^2) - \ell^2}} \dfrac{k-\ell}{\sqrt{k(k+\sigma^2)}} \\
&= \arccos\left(\dfrac{\ell}{\sqrt{k(k+\sigma^2)}}\right) - \dfrac{k-\ell}{\sqrt{k(k+\sigma^2) - \ell^2}}.
\end{align*}
We thus obtain a somewhat simple lower bound on the mean of $Z_i$
\begin{align*}
\mathbb{E}[Z_i] &\geqslant \dfrac{1}{\pi} \dfrac{k-\ell}{\sqrt{k(k+\sigma^2)-\ell^2}}.
\end{align*}
We now have all the ingredients to upper-bound the probability that $w$ performs strictly better than $w^\star$, in other words that $\Delta(w,w^\star) := \sum_{i=1}^n Z_i < 0$. Applying Hoeffding's inequality for independent random variables supported on $[-1, 1]$, we have for any $t >0$
\begin{align*}
\mathbb{P}\left( \sum_{i=1}^n (Z_i - \mathbb{E}[Z_i] ) < -n t\right) &\leqslant \exp \left( - \dfrac{n t^2}{2}\right), \end{align*}
and taking $t = \mathbb{E}[Z]$, which is non negative for $\ell < k$ because arccos is decreasing on $[0,1]$, leads to
\(
\mathbb{P}\left( \Delta(w,w^\star) < 0\right) \leqslant \exp \left( - \frac{n}{2} \mathbb{E}[Z_i]^2 \right).
\)
Substituting in the previous expression our lower bound for the mean $\mathbb E[Z_i]$ gives the desired result. 
\hfill\halmos
\endproof

{ \paragraph{Remark:} In the absence of noise ($\sigma=0$), the bound in Lemma \ref{large_dev} can be improved. Indeed, the truth makes no mistakes ($\hat y_i(w^\star) = y_i, \;  \forall i$) and $\Delta(w,w^\star) \geqslant 0$ for any $w$. More precisely, here $\sum_{i=1}^n \textbf{1}\left(\hat y_i(w) \neq y_i \right)$ is a binomial random variable with parameters $n$ and $q(\ell; k, 0)$. Therefore, 
 \begin{align*}
 \mathbb{P} \left( \Delta(w,w^\star) \leqslant 0 \right) = \mathbb{P} \left( \Delta(w,w^\star) = 0 \right) = \left( 1 - \dfrac{1}{\pi}  \arccos\left( \tfrac{\ell}{k} \right) \right)^n. 
 \end{align*}
Using concavity of the logarithm and the arccos function successively yields the upper bound
\[ \mathbb{P}\left( \Delta(w,w^\star) < 0\right) \leqslant \exp \left( - \dfrac{n}{\pi} \sqrt{\dfrac{k-\ell}{k+\ell}} \right). \]
However, such a refinement eventually modifies the result in Theorem \ref{sufficient} by a constant multiplicative factor only.}

\subsection{Proof Theorem \ref{sufficient}}
\proof{Proof of Theorem \ref{sufficient}}
We are interested in bounding the probability that the binary classifier with minimal empirical misclassification rate is any other than $w^\star$. We can characterize the probability of such event as $\mathbb{P}\left(\exists w \neq w^\star \mbox{ s.t. } \Delta(w,w^\star) \leqslant 0 \right)$. Evidently,
  \[
    \mathbb{P}\left(\exists w \neq w^\star \mbox{ s.t. } \Delta(w,w^\star) \leqslant 0 \right) = \textstyle \sum_{\ell \in \{0, \dots, k-1\}} \mathbb{P}\left(\exists w \mbox{ s.t. } w^\top w^\star = \ell, ~ \Delta(w,w^\star) \leqslant 0 \right).
  \]
Recall that there are exactly $\binom{k}{\ell} \binom{p-k}{k-l}$ distinct binary classifiers $w$ with accuracy $w^\top w^\star = \ell$.
Combining a union bound and the bound from Lemma \ref{large_dev} yields
\begin{align}
  \mathbb{P}\left(\exists w \neq w^\star \mbox{ s.t. } \Delta(w,w^\star) \leqslant 0 \right) 
  &\leqslant \sum_{\ell=0}^{k-1} \binom{k}{\ell} \binom{p-k}{k-\ell} \exp\left(  -n \frac{(k-\ell)^2}{2 \pi^2 (k (k+\sigma^2)-\ell^2)} \right), \\
  & \label{eq:ineq} \leqslant k  \cdot \max_{\ell=0,...,k-1} \binom{k}{\ell} \binom{p-k}{k-\ell} \exp\left( -n \frac{(k-\ell)^2}{2 \pi^2 (k (k+\sigma^2)-\ell^2)}\right).
\end{align}
In order for the previous error probability to be bounded away from one, it suffice to take $n$ greater than a threshold $T$
\begin{align*}
n > T : = \max_{\ell=0,...,k-1} \frac{2 \pi^2 (k (k+\sigma^2)-\ell^2)}{(k-\ell)^2} \left[ \log k + \log \binom{k}{\ell} + \log \binom{p-k}{k-\ell}\right].
\end{align*}
We can obtain a more interpretable sufficient condition by upper-bounding the threshold $T$. Assuming $p \geqslant 2 k$, $\binom{k}{\ell} = \binom{k}{k-\ell} \leqslant \binom{p-k}{k-\ell}$ and $k \leqslant \binom{p-k}{k-\ell}$, so that
\begin{align*}
T &\lesssim \max_{\ell=0,...,k-1} \frac{k (k+\sigma^2)-\ell^2}{(k-\ell)^2} \log \binom{p-k}{k-\ell},
\end{align*}
where $\lesssim$ signifies that the inequality holds up to a multiplicative factor. {  Since $\log \binom{p-k}{k-\ell} \lesssim (k-\ell) \log \left(\dfrac{p-k}{k-\ell}\right)$}, we now have
\begin{align*}
  T & \lesssim \max_{\ell=0,...,k-1} \frac{k (k+\sigma^2)-\ell^2}{(k-\ell)} \log \left(\dfrac{p-k}{k-\ell}\right), \\
\end{align*}
The maximum over $\ell$ in right hand side of the previous inequality occurs when $\ell = k -1$. The previous observation yields hence that $T  \lesssim (2+\sigma^2) k  \log (p-k)$. Finally, it is easy to verify that when $n$ exceeds some threshold value $n_0 \geqslant T$, the inequality \eqref{eq:ineq} yields
\begin{align*}
  \mathbb{P}\left(\exists w \neq w^\star \mbox{ s.t. } \Delta(w,w^\star) \leqslant 0 \right) & \leqslant \max_{\ell=0,...,k-1} \exp\left( - \frac{(k-\ell)^2}{2 \pi^2 (k (k+\sigma^2)-\ell^2)} (n-n_0)\right), \\
  & \leqslant \exp\left( -\frac{n - n_0}{2 \pi^2 k (2+\sigma^2)}\right).
\end{align*}
\hfill\halmos
\endproof
\end{APPENDICES}

\newpage
\footnotesize
\bibliographystyle{informs2014}
\bibliography{biblio}
\end{document}